\documentclass[11pt]{article}
\usepackage [dvips]{graphics}
\usepackage[centertags]{amsmath}
\usepackage{amsfonts}
\usepackage{amssymb}
\usepackage{amsthm}
\usepackage{newlfont}
\usepackage{stmaryrd}
\usepackage{mathrsfs}
\usepackage[colorlinks,
            urlcolor=black,
            linkcolor=blue,
            anchorcolor=blue,
            citecolor=blue
            ]{hyperref}
\usepackage{CJK}
\usepackage{indentfirst}
\usepackage{cases}
\usepackage{bookmark}
\usepackage{graphicx}
\usepackage{pdfpages}
\usepackage{setspace}
\usepackage{abstract}

\theoremstyle{plain}
\newtheorem{thm}{Theorem}[section]
\newtheorem{cor}[thm]{Corollary}
\newtheorem{lem}[thm]{Lemma}
\newtheorem{prop}[thm]{Proposition}
\newtheorem{exam}[thm]{Example}
\newtheorem{rem}[thm]{Remark}

\newtheorem{ass}[thm]{Assumption}

\def\sqr#1#2{{\vcenter{\vbox{\hrule height.#2pt
              \hbox{\vrule width.#2pt height#1pt \kern#1pt \vrule
width.#2pt}
              \hrule height.#2pt}}}}
\def\dbR{{\mathbb{R}}}

\def\z{\zeta}

\def\l{\lambda}

\def\3n{\negthinspace \negthinspace \negthinspace }
\def\2n{\negthinspace \negthinspace }
\def\1n{\negthinspace }

\def\no{\noindent}

\def\ms{\medskip}
\def\bs{\bigskip}

\def\({\Big (}
\def\){\Big )}
\def\[{\Big[}
\def\]{\Big]}

\def\be{\begin{equation}}
\def\bel{\begin{equation}\label}
\def\ee{\end{equation}}
\def\bea{\begin{eqnarray}}
\def\eea{\end{eqnarray}}
\def\bt{\begin{theorem}}
\def\et{\end{theorem}}
\def\bc{\begin{corollary}}
\def\ec{\end{corollary}}
\def\bl{\begin{lemma}}
\def\el{\end{lemma}}
\def\bp{\begin{proposition}}
\def\ep{\end{proposition}}
\def\br{\begin{remark}}
\def\er{\end{remark}}
\def\ba{\begin{array}}
\def\ea{\end{array}}
\def\bd{\begin{definition}}
\def\ed{\end{definition}}

\makeatletter
   
   \@addtoreset{equation}{section}
\makeatother

\begin{document}

\title{\bf Eigenvalues  of stochastic Hamiltonian systems with boundary conditions and its application}

\author{ Guangdong Jing\thanks{ Research partially supported by the National Natural Science Foundation of China (No.11871308), {E-mail:} {jingguangdong@mail.sdu.edu.cn}.}
\quad Penghui Wang \thanks{Research partially supported by the National Natural Science Foundation of China (No.11471189, No.11871308), {E-mail:} {phwang@sdu.edu.cn}. \ms}
\\ \\ School of Mathematics, Shandong University, Jinan 250100,  China}

\maketitle
\vspace{-1.2cm}

\begin{abstract}
\noindent\textbf{Abstract:} In this paper we solve the eigenvalue problem of stochastic Hamiltonian system with boundary conditions. Firstly, we extend the results in S. Peng \cite{peng} from time-invariant case to time-dependent case, proving the existence of a series of eigenvalues $\{\lambda_m\}$ and construct corresponding eigenfunctions. Moreover, the order of growth for these $\{\lambda_m\}$ are obtained: $\lambda_m\sim m^2$, as $m\rightarrow+\infty$. As applications, we give an explicit estimation formula about the statistic period of solutions of Forward-Backward SDEs. Besides, by a meticulous example  we show  the subtle situation in time-dependent case  that some eigenvalues appear when the solution of the associated Riccati equation does not blow-up, which does not happen in time-invariant case.
\end{abstract}

\bs

\no{\bf Keywords}:  Eigenvalue problem; Forward-Backward SDE; Stochastic Hamiltonian system; Monotonicity condition; Statistic period; Method of decoupling; Riccati equation.

\bs

\no{\bf 2020 AMS MSC}:  60H10; 34B99; 34F05; 34L15.

\section{Introduction}
Let $(\Omega,\mathscr{F}, \mathbb{F}, \mathbb{P})$
be a complete filtered probability space, on which a standard $1$-dimensional
Brownian motion $B=\{B_t\}_{t\ge0}$ is defined, and
$\mathbb{F}=\{\mathscr{F}_t\}_{t\geq 0}$ is the natural filtration of $B$
augmented by all the $\mathbb{P}$-null sets in $\mathscr{F}$.
Let $T>0$ be any fixed time horizon.

In this paper we consider the eigenvalue problem of stochastic Hamiltonian system with time-dependent coefficients with boundary conditions. In general, it can be formulated as finding $\l\in\dbR$ such that the following system   has nontrivial solutions:
\begin{equation}  \label{evp1}
\left\{
\begin{aligned}
& \mathrm{d}x_t=\partial_yh^\l(x_t,y_t,z_t)\mathrm{d}t
   +\partial_zh^\l(x_t,y_t,z_t)\mathrm{d}B_t, \indent   t\in[0,T],    \\
&  -\mathrm{d}y_t=\partial_xh^\l(x_t,y_t,z_t)\mathrm{d}t
   -z_t\mathrm{d}B_t,      \indent   t\in[0,T],                       \\
& x(0)=0,\indent   y(T)=0,
\end{aligned}
\right.
\end{equation}
where $h^\lambda=h+\lambda \bar{h}$ and $h,\bar{h}:\mathbb{R}^n\times\mathbb{R}^n\times\mathbb{R}^n\rightarrow\mathbb{R}$ belong to $C^1$ with $\partial_xh=\partial_x\bar{h}=\partial_yh=\partial_y\bar{h}=\partial_zh=\partial_z\bar{h}=0$ for $(x,y,z)=(0,0,0)$.
The above problem is a stochastic counterpart of the classical eigenvalue problem of mechanic systems (see Remark \ref{parallel-increase-speed}).
The latest progress in this topic can be found in \cite{JW,wuzhen-wanghaiyang}.

The stochastic Hamiltonian system was originally introduced in the optimal control theory as a necessary condition for optimality. \cite{Bensoussan,Bismut,peng2} are pioneer results in this topic.  The eigenvalue problem of stochastic Hamiltonian system is closely related to the solvability of Forward-Backward Stochastic Differential Equations (FBSDEs in short), about which there are mainly three methods in literature.
Firstly, the Contraction Mapping method \cite{Antonelli,PT,zhangjianfeng2}  which is a local result to some extent and uniform estimation is necessary if one wants to  obtain global results. Secondly, the Decoupling method \cite{Delarue,MPY,MWZZ,Yongjiongmin3,zhangjianfeng2}, which always appears whenever FBSDEs and PDEs are linked. Thirdly, the Continuation method  \cite{hupeng,wupeng,Y} based on the Monotonicity Condition.
See also the comments and references in the monograph \cite[Chapter 8]{zhangjianfeng1}.


In addition to the monotonicity condition, the method of decoupling for linear FBSDEs also plays an important role in this paper (see Lemma \ref{decouple-lemma} and the comments there). A similar idea appears in \cite[Section 5]{Y2} and  \cite[Chapter $2$, $\S 4$]{MY}.

Note that the eigenvalue problem for the stochastic Hamiltonian system with boundary conditions is different from and much more complicated than its counterpart in deterministic framework. Since the eigenfunctions in stochastic case should be progressively measurable, most of the techniques in dealing with deterministic eigenvalue problem  do NOT work anymore. In particular, it is  totally different from the associated  deterministic system by taking expectation directly, which can be observed  from the example in Appendix \ref{examp-deny-naive-expectation}.

In \cite{peng}, S. Peng considered the following eigenvalue problem,
\begin{equation}  \label{evp2}
\left\{
\begin{aligned}
& \mathrm{d}x_t=[H_{21}^\l x_t+ H_{22}^\l y_t
   +H_{23}^\l z_t]\mathrm{d}t+[H_{31}^\l x_t+H_{32}^\l y_t
   +H_{33}^\l z_t]\mathrm{d}B_t,\indent  t\in[0,T],                \\
& -\mathrm{d}y_t=[H_{11}^\l x_t+H_{12}^\l y_t
   +H_{13}^\l z_t]\mathrm{d}_t-z_t \mathrm{d} B_t,    \indent  t\in[0,T],        \\
& x(0)=0,\indent   y(T)=0,
\end{aligned}
\right.
\end{equation}
where $H^\l=H-\l \bar{H}$,
$$
  H=\begin {bmatrix}
  H_{11}&H_{12}&H_{13}\\
  H_{21}&H_{22}&H_{23}\\
  H_{31}&H_{32}&H_{33}
  \end{bmatrix}     \quad\text{ and }\quad
  \bar{H}=\begin {bmatrix}
  \bar{H}_{11}&\bar{H}_{12}&\bar{H}_{13}\\
  \bar{H}_{21}&\bar{H}_{22}&\bar{H}_{23}\\
  \bar{H}_{31}&\bar{H}_{32}&\bar{H}_{33}
  \end{bmatrix}
$$
are constant matrices, moreover,
$H^\l_{ij}=H_{ij}-\l \bar{H}_{ij}$, $H_{ij}=H_{ij}^\top$, and
$\bar{H}_{ij}=\bar{H}_{ij}^\top$, $i,j=1,2,3$.

For 1-dimensional case,  when
$$
 \bar{H}=\begin {bmatrix}
  0&0&0\\
  0&H_{22}&0\\
  0&0&0
  \end{bmatrix},
$$
S. Peng proved the following

\begin{thm}[\cite{peng}, Theorem 3.2]\label{Theorem-Peng}
Under assumption \eqref{moncondition} and condition $H_{23}=-H_{33}H_{13}$, all the eigenvalues $\{\lambda_m\}$ of \eqref{evp2} are positive and $\lambda_m\to+\infty$ as $m\to+\infty$.
Moreover, all the eigenspaces associated with  each $\lambda_m$ are $1$-dimensional.
\end{thm}

Based on the above theorem,  by rather exhaustive analysis, we obtained in \cite{JW} the following
\begin{thm}[\cite{JW}, Theorem 1.3]\label{thm-const-increa-intro}
Under the same assumptions  in Theorem \ref{Theorem-Peng}, let $\{\lambda_m\}$ be the eigenvalues, then
\begin{equation*}
\lambda_m= O(m^2), \indent   \text{as}\ m\to+\infty.
\end{equation*}
In detail,
\begin{equation*}
\frac{\pi^2}{-2H_{11}H_{22}T^2} \le \mathop{\underline{\lim}}_{m\to+\infty} \frac{\lambda_m}{m^2}\leq \mathop{\overline{\lim}}_{m\to+\infty}  \frac{\lambda_m}{m^2}\le \frac{4\pi^2}{-H_{11}H_{22}T^2}.
\end{equation*}
\end{thm}

\begin{rem}\label{parallel-increase-speed}
Recall the following eigenvalue problem of a special Hamiltonian system in deterministic framework
\begin{equation*}
\left\{
\begin{aligned}
&\frac{\mathrm{d}x}{\mathrm{d}t}=\lambda y(t),\indent   t\in[0,T],  \\
&-\frac{\mathrm{d}y}{\mathrm{d}t}=x(t),\indent   t\in[0,T],  \\
&x(0)=0,\indent  y(T)=0.
\end{aligned}
\right.
\end{equation*}
Its eigenvalues are
$\left(\frac{\left(2m-1\right)\pi}{2T}\right)^2,\ m=1,2,3, \cdots$.
From the theoretical value aspect, the conclusions  in Theorem \ref{thm-const-increa-intro} and Theorem \ref{thm-increa-speed-introd} for  stochastic systems  can be considered as an analogue.
\end{rem}

As a corollary of Theorem \ref{thm-const-increa-intro}, we have
\begin{prop}[\cite{JW}, Corollary 1.5]\label{con-prop-lambd-estima}
Let $\lambda$ be an eigenvalue of the stochastic Hamiltonian system in Theorem \ref{thm-const-increa-intro}. Then for sufficiently large $m$, if
\begin{equation*}
  \lambda<\frac{m^2 \pi^2}{-2H_{11}H_{22}T^2},\indent \left(\text{resp.} \quad \lambda>\frac{4m^2\pi^2}{-H_{11}H_{22}T^2} \right)
\end{equation*}
the statistic period  of the associate eigenfunctions (i.e., the solutions of FBSDEs) is less (resp. greater) than $m$.
\end{prop}

In this paper, we  study the eigenvalue problem \eqref{evp2} in $1$-dimensional case with time-dependent coefficients:
\begin{equation}\label{main-eigenvalue-problem}
\left\{
\begin{aligned}
&\mathrm{d}x_t=\left[H_{21}x_t+(H_{22}-\lambda h_{22})y_t+H_{23}z_t\right]\mathrm{d}t   +\left[H_{31}x_t+H_{32}y_t+H_{33}z_t\right]\mathrm{d}B_t, \ \    t\in[0,T],       \\
&-\mathrm{d}y_t=\left[H_{11}x_t+H_{12}y_t
  +H_{13}z_t\right]\mathrm{d}t-z_t \mathrm{d}B_t,    \indent    t\in[0,T],            \\
&x(0)=0,\indent  y(T) =0,
\end{aligned}
\right.
\end{equation}
where $H_{ij}, h_{22} \in C[0,T], i,j=1,2,3$,  $H_{23}(t)=-H_{33}(t)H_{13}(t), h_{22}(t)<0, \forall t \in [0,T]$.

The following  theorems, detailed content of which are given in Theorem \ref{general-lambda-exist}, Theorem \ref{increase-ratio-function-general} and Theorem \ref{thm-estim-period-coef},  are the  main results in this paper.
The technical ingredients of their proof consist of Legendre transformation, the method of decoupling for FBSDEs, several concrete kinds of comparison theorems, constructing proper auxiliary systems, analyzing the blow-up time of associated Riccati equations, and many other elementary tools in ODE theory.

\begin{thm}\label{thm-increa-speed-introd}
Let $\lambda_b$ be a positive constant defined in \eqref{lambda-b-h22-uniform-posit}. Under Assumption \ref{assumption-1d-general-perturbation}, there exists $\{\lambda_m\}_{m=1}^\infty \subset (\lambda_b, +\infty)$, all those eigenvalues  of problem (\ref{generel-function-eigen-problem}) contained in $(\lambda_b, +\infty)$, satisfying $\lambda_m\rightarrow+\infty$ as \ $m\rightarrow +\infty$.
Besides, the eigenspace associated with each $\lambda_m$ is of  $1$ dimension.
Moreover,
\begin{equation*}
\lambda_m=O(m^2), \indent \ \text{as} \ m\rightarrow +\infty.
\end{equation*}
\end{thm}

It is worth noting that the results in Theorem \ref{thm-increa-speed-introd}, in addition to its theoretical value, together with Proposition \ref{con-prop-lambd-estima}, can be utilized  to estimate the statistic period of solutions of FBSDEs   directly by its time-dependent coefficients and time duration.
\begin{thm}
  Let $\lambda_m$  be an eigenvalue in Theorem \ref{thm-increa-speed-introd}, then for sufficiently large $m\in\mathbb{N}_+$,
\begin{equation*}
\frac{\hat{H}_{22}-\underline{H}_{22}}{\check{h}_{22}}  + \frac{\pi^2 m^2}{-2\hat{H}_{11}\check{h}_{22} T^2} \le \lambda_m \le \frac{4\pi^2 m^2}{-\check{H}_{11}\hat{h}_{22} T^2}.
\end{equation*}
Therefore, if
\begin{equation*}
 \lambda>\frac{4\pi^2 m^2}{-\check{H}_{11}\hat{h}_{22} T^2},\indent
           \left(resp. \quad  \lambda< \frac{\hat{H}_{22}-\underline{H}_{22}}{\check{h}_{22}}
                                        + \frac{\pi^2 m^2}{-2\hat{H}_{11}\check{h}_{22} T^2} \right)
\end{equation*}
 the statistic period of the eigenfunctions associated with $\lambda$ is greater (resp. less) than $m$.
\end{thm}

\begin{rem}
The eigenvalue problem for the stochastic Hamiltonian system with time-dependent coefficients is much more complicated than the time-independent coefficients case. Because in the latter case,  by \cite{peng}, all the eigenvalues come from the blow-up of the associated Riccati equation and dual Riccati equation. However, for the time-dependent coefficients case,  the example in Section \ref{illustrating unusual example} shows  that some eigenvalues appear when the solution of the Riccati equation does not blow-up.
\end{rem}



In this paper, $m$ is used to denote the second part of solution $(k,m)$ to the Riccati equation \eqref{riccati-lemma-n-sep-h22-k}  and the index of eigenvalues $\{\lambda_m\}_{m=1}^{+\infty}$.
$n$ is used to denote both the dimension of Hamiltonian system and the index in \eqref{the-n-exist-t2npl1=0}.
We believe that it will not cause  ambiguity.

Similar results of this paper  hold for the eigenvalue problem of stochastic Hamiltonian system driven by Poisson processes. In order to keep this paper in a suitable length, we postpone those results to another paper.

The paper is organized as follows.  In Section \ref{Problem of the first positive eigenvalue}, several lemmata are introduced which will be used repeatedly.
In Section \ref{section-formulation}, the main problem in this paper is formulated.
In Section \ref{section-existence-eigenvalue},  we prove the existence of all the eigenvalues located in $(\lambda_b, +\infty)$ and then all the eigenvalues in $\mathbb{R}$ under some sharper conditions in Section \ref{examp-find-all-eigen-sec-ass}. Moreover, the increasing order of these $\{\lambda_m\}_{m=1}^{+\infty}$ are studyed in Section \ref{section-increasing-order}.
Most importantly, apart from its theoretical value, as an interesting application, this result can be utilized to obtain an estimation about statistic period of solutions of FBSDEs, which is investigated in Section \ref{section-estimation-period}. In Section \ref{illustrating unusual example}, by a concrete example, we show how the eigenvalue problem of stochastic Hamiltonian system with time-dependent coefficients is far more subtler than its time-independent counterpart. At last, several examples, the proof of several lemmata,   the review of the viewpoint from functional analysis and Legendre transformation  are gathered in the Appendix.



\section{Preliminaries }\label{Problem of the first positive eigenvalue}

\subsection{Comparison theorems}
Let $S_n$ denote the set of all $n\times n$ symmetric matrices, and $S_n^+$ the set of all nonnegative matrices in $S_n$.
For $K\in S_n^+$, $K\ge0$ means that $K$ is positive definite while $K>0$ strictly positive definite. 
Given two nonlinear $S_n$-value ODEs: for $i=1,2$,
\begin{equation} \label{equation-compare-riccati-ori}
\left\{
\begin{aligned}
&-\frac{\mathrm{d}K_i}{\mathrm{d}t}=K_i A(t)+A^\top(t)K_i+C^\top K_iC+R_i(t)+K_iN_i(t)K_i  \\
&\indent\indent\ \ \  +\left(B(t)+K_iD(t)\right)F_i(K_i)\left(B(t)+K_iD(t)\right)^\top, \indent t\le T,\\
&K_i(T) =Q_i,
\end{aligned}
\right.
\end{equation}
where $A,B,C,D\in C([0,T],\mathbb{R}^{n\times n})$, $R_i,N_i\in C([0,T],S_n)$, and $F_i:S_n\mapsto S_n$, $i=1,2$ are locally Lipschitz.

\begin{lem}[\cite{peng}, Lemma 8.1]  \label{elementary-compare}
Denote by $K$  the solution to
\begin{equation}
\left\{
\begin{aligned}
&-\frac{\mathrm{d}K}{\mathrm{d}t}=A^\top(t)K+KA(t)+C^\top(t)KC(t)+R_1(t),\indent t\le T,  \\
&K(T)=Q_1.
\end{aligned}
\right.
\end{equation}
If $Q_1\in S_n^+$, and $R_1(t)\in S_n^+,\ t\le T$, then $K(t)\in S_n^+$.
Moreover, if $Q_1>0$, or $R_1(t)>0,\ t\le T$, then $K(t)>0,\ t< T$.
\end{lem}

\begin{lem}[\cite{peng}, Lemma 8.2]\label{comparison theorem}
Assume that $K_i,\ i=1,2$, are the solutions to (\ref{equation-compare-riccati-ori}) separately and
  $$Q_1\ge Q_2,\indent  R_1(t)\ge R_2(t),\indent  N_1(t)\ge N_2(t),\indent  \forall t\in[0,T];$$
  $$F_1(K)\ge F_1(K'),\indent \forall K\ge K'; \indent \ F_1(K)\ge F_2(K),\indent \forall K\in S_n.$$
Then
  $$K_1(t)\ge K_2(t).$$
\end{lem}

\begin{lem} \label{const-term-posit-equation-posit}
Assume that $\psi_1, \psi_2, \psi_3 \in C([0,T],\mathbb{R})$, and $\psi_1\ge c>0$ $(\psi_1\le -c<0, resp.)$.
Denote by $\Phi$ the solution to the following equation:
\begin{equation*}
\left\{
\begin{aligned}
&-\frac{\mathrm{d}\Phi}{\mathrm{d}t}=\psi_1+\psi_2\Phi+\psi_3\Phi^2,\indent t\le T, \\
&\Phi(T)=0.
\end{aligned}
\right.
\end{equation*}
Then $\Phi(t)>0,\ t<T$ $(\Phi(t)<0,\ t<T, resp.)$.
\end{lem}

\subsection{Decoupling method for linear FBSDEs}
By the method introduced in \cite{peng}, every linear FBSDE:
\begin{equation} \label{stochastic-Hamilton}
\left\{
\begin{aligned}
&\mathrm{d}x_t=[H_{21}x_t+ H_{22}y_t+ H_{23}z_t]\mathrm{d}t + [H_{31}x_t+ H_{32}y_t+ H_{33}z_t]\mathrm{d}B_t,  \indent t\in[T_1,T_2],\\
&-\mathrm{d}y_t=[H_{11}x_t+H_{12}y_t + H_{13}z_t]\mathrm{d}t-z_t \mathrm{d}B_t, \indent t\in[T_1,T_2],  \\
&x(T_1)=x_0,\indent y(T_2)=K_{T_2}x(T_2),
\end{aligned}
\right.
\end{equation}
corresponds to a Riccati type ODE:
\begin{numcases}{}
-\frac{\mathrm{d}K}{\mathrm{d}t}=K( H_{21}+ H_{22}K+ H_{23}M) + H_{11}+  H_{12}K+ H_{13}M, \indent t\in[T_1,T_2], \label{riccati-lemma-n-sep-1}\\
M=K( H_{31}+H_{32}K+H_{33}M),\indent t\in[T_1,T_2],  \label{riccati-lemma-n-sep-2}   \\
K(T_2) =K_{T_2}\in S_n,  \label{riccati-lemma-n-sep-3}
\end{numcases}
where $(K,M)\in C^1([T_1, T_2];S_n^+)\times L^{\infty}([T_1, T_2];\mathbb{R}^{n\times n}), [T_1, T_2]\subset [0, T]$.

Riccati equation (\ref{riccati-lemma-n-sep-1})-(\ref{riccati-lemma-n-sep-3}) is introduced in a fantastic manner: it transfers the fully-coupled FBSDEs (\ref{stochastic-Hamilton}) into decoupled one. Lemma \ref{decouple-lemma} depicts the detail.

The following lemma is a generalization of \cite[Lemma 4.2]{peng}, from constant coefficients case to time-dependent coefficients case. However, since the proof is standard, we put it in appendix.

\begin{lem} \label{decouple-lemma}
Assume that on some interval $[T_1,T_2]\subseteq (-\infty,T]$, Riccati equation (\ref{riccati-lemma-n-sep-1})-(\ref{riccati-lemma-n-sep-3})
has a solution $(K,M)$.
Then stochastic Hamiltonian system with boundary conditions (\ref{stochastic-Hamilton}) has an \emph{explicit solution}:
$$ (x(t),y(t),z(t))=\left(x(t),K(t)x(t),M(t)x(t)\right),\indent t\in[T_1,T_2],  $$
where $x(t)$ is solved by
\begin{equation} \label{forward-sde-decouple-lemma}
\left\{
\begin{aligned}
&\mathrm{d}x_t=[H_{21}+ H_{22}K+ H_{23}M]x_t\mathrm{d}t   \\
&\indent\ \ \  + [ H_{31}+ H_{32}K+ H_{33}M]x_t\mathrm{d}B_t,   \indent    t\in [T_1,T_2],   \\
&x(T_1)=x_0.
\end{aligned}
\right.
\end{equation}
Further, for $t\in [T_1,T_2]$, if $\det(I_n-K(t)H_{33}(t))\neq0$, or more weakly, there is a constant $c>0$, such that
\begin{equation}\label{weak-unique-condition}
  (I_n-K(t)H_{33}(t))^\top(I_n-K(t)H_{33}(t))\ge c (H_{13}(t)+K(t)H_{23}(t))^\top(H_{13}(t)+K(t)H_{23}(t)),
\end{equation}
then the solution to (\ref{stochastic-Hamilton}) is unique.
\end{lem}

\subsection{Investigation of the coefficients of the derived Riccati equations}
Since (\ref{riccati-lemma-n-sep-2}) can be rewritten as $[I_n-K(t)H_{33}(t)]M(t)=K(t)(H_{31}(t)+H_{32}(t)K(t))$, condition $\det\left(I_n-K(t)H_{33}(t)\right)\neq0,\ \forall t\in [T_1, T_2]$ is necessary for the unique existence of solution $(K,M)$ to (\ref{riccati-lemma-n-sep-1})-(\ref{riccati-lemma-n-sep-3}).

\begin{lem}\label{analysis-F0K}
Let $\beta_1>\beta>0$ such that $-\beta_1 I_n \le H_{33}(t)\le-\beta I_n<0$, $t\in[0,T]$. Then
\begin{enumerate}
  \item For $K\in S_n$ and $K>-\frac{1}{2\beta_1}I_n$, there is a constant $c>0$, such that
        $$\left\|\left(I_n-KH_{33}(t)\right)^{-1}\right\|\le c, \ \ \ \ \forall t\in [0, T].$$
  \item For $K>0$, we have $(F_0(K))^\top=F_0(K)$ and
    \begin{eqnarray} \label{F0-K-bound}
       0\le F_0(K)=(I_n-KH_{33}(t))^{-1}K \le -H_{33}^{-1}(t)\le \frac{1}{\beta}I_n,\ \ \ \ \forall  t\in [0, T].
    \end{eqnarray}
  \item For $K_1,K_2\in S^+_n,\ K_1\ge K_2$, we have $F_0(K_1)\ge F_0(K_2)$, $\forall t\in [0, T]$.
\end{enumerate}
\end{lem}

\begin{proof}
1.  Since $$-\beta_1 I_n \le H_{33}(t)\le-\beta I_n<0, \indent t\in[0,T],$$ $H_{33}(t)$ is invertible and $$0<\frac{1}{\beta_1}I_n\le -H_{33}^{-1}(t)\le\frac{1}{\beta}I_n.$$
It follows that
\begin{eqnarray}\label{equ1}
\left(I_n-KH_{33}(t)\right)^{-1}=-H_{33}^{-1}(t)\left(-H_{33}^{-1}(t)+K\right)^{-1}.
\end{eqnarray}
Noting that $K>-\frac{1}{2\beta_1}I_n$,
then $0\le\left(-H_{33}^{-1}(t)+K\right)^{-1}\le2\beta_1 I_n$.
Then for $\forall t \in[0,T],$
$$\left\|(I_n-KH_{33}(t))^{-1}\right\|\leq \left\|-H_{33}^{-1}(t)\right\|\left\|\left(-H_{33}^{-1}(t)+K\right)^{-1}\right\|\leq \frac{2\beta_1}{\beta}.$$

2. For $K>0$, $$F_0(K)=\left(I_n-KH_{33}(t)\right)^{-1}K=\left(K^{-1}-H_{33}(t)\right)^{-1}.$$
Then $(F_0(K))^\top=F_0(K)$.
Besides, by  $K^{-1}\ge0$ and $-H_{33}(t)\ge \beta I_n>0,t\in[0,T]$,
we obtain
$$0\le F_0(K)=\left(K^{-1}-H_{33}(t)\right)^{-1}\le \frac{1}{\beta}I_n.$$


3. At first, assume that $K_1\ge K_2>0$, then $K_1^{-1}\leq K_2^{-1}$, and hence
$$0<K_1^{-1}-H_{33}(t)\leq K_2^{-1}-H_{33}(t).$$
It follows that
\begin{equation}\label{equ2}
F_0(K_1)=\left(K_1^{-1}-H_{33}(t)\right)^{-1}\ge \left(K_2^{-1}-H_{33}(t)\right)^{-1}=F_0(K_2).
\end{equation}
For general case, $K_1, K_2\in S_n^+$ and $K_1\ge K_2$, by (\ref{equ1}), $I_n-K_iH_{33}$ are invertible, and hence $F_0(K_i)$ can be defined. By \eqref{equ2},
$$F_0(K_1+\epsilon)=\left((K_1+\epsilon)^{-1}-H_{33}(t)\right)^{-1}\ge \left((K_2+\epsilon)^{-1}-H_{33}(t)\right)^{-1}=F_0(K_2+\epsilon).$$
It follows that
$$F_0(K_1)=\lim\limits_{\epsilon\to 0^+}F_0(K_1+\epsilon)\ge \lim\limits_{\epsilon\to 0^+}F_0(K_2+\epsilon)=F_0(K_2).$$
Besides, $F_0(0)=0$, then $F_0(K)$ is a mapping from $S_n^+$ to $S_n^+$,
\end{proof}

\subsection{Monotonicity condition}
For linear FBSDEs with constant coefficients \eqref{evp2} without perturbation, i.e., $\bar{H}=0$,  the classical monotonicity condition is as follows:
\begin{eqnarray}
\begin {bmatrix}
-H_{11}&-H_{12}&-H_{13}\\
H_{21}&H_{22}&H_{23}\\
H_{31}&H_{32}&H_{33}
\end{bmatrix}
\leqslant-\beta I_{3n},\end{eqnarray}
where $\beta>0$ is a constant. That is, for $(x,y,z)\in \mathbb R^n\times\mathbb R^n\times\mathbb R^n$,
\begin{eqnarray}\label{moncondition}
\begin{bmatrix}
x^\top&y^\top&z^\top
\end{bmatrix}
 \begin {bmatrix}
-H_{11}&-H_{12}&-H_{13}\\
H_{21}&H_{22}&H_{23}\\
H_{31}&H_{32}&H_{33}
\end{bmatrix}
\begin {bmatrix}
x\\y\\z
\end{bmatrix}
\leq -\beta\left(\|x\|^2+\|y\|^2+\|z\|^2\right).
\end{eqnarray}
By taking $(x,y,z)=(x,0,0), (0,y,0), (0,0,z)$  in \eqref{moncondition},
\begin{eqnarray}  \label{moncond-H11H22H33}
H_{11}\ge\beta I_n, \indent  H_{22}\le-\beta I_n, \indent H_{33}\le-\beta I_n.
\end{eqnarray}
Besides,
\begin{eqnarray}  \label{moncond-H22H23H33-1H32}
H_{22}-H_{23}H_{33}^{-1}H_{32}<0.
\end{eqnarray}
\emph{In this paper, we always assume that the monotonicity condition \eqref{moncondition} holds true.}

\section{Formulation}\label{section-formulation}
In this section, we formulate the eigenvalue problem:
\begin{equation}\label{generel-function-eigen-problem}
\left\{
\begin{aligned}
&\mathrm{d}x_t=\left[H_{21}x_t+(H_{22}-\lambda h_{22})y_t+H_{23}z_t\right]\mathrm{d}t   +\left[H_{31}x_t+H_{32}y_t+H_{33}z_t\right]\mathrm{d}B_t, \ \    t\in[0,T],       \\
&-\mathrm{d}y_t=\left[H_{11}x_t+H_{12}y_t
  +H_{13}z_t\right]\mathrm{d}t-z_t \mathrm{d}B_t,    \indent    t\in[0,T],            \\
&x(0)=0,\indent  y(T) =0.
\end{aligned}
\right.
\end{equation}

\begin{ass} \label{assumption-1d-general-perturbation}
Assume that $n=1$ and $H_{ij}\in C([0,T],\mathbb{R})$, $i,j=1,2,3$. Besides, $H$ satisfy (\ref{moncondition}) uniformly for $t\in[0,T]$.
Moreover,
\begin{eqnarray} \label{h23t=-h33th13t}
H_{23}(t)=-H_{33}(t)H_{13}(t),\indent   t\in[0,T].
\end{eqnarray}
Besides, $h_{22}\in C([0,T],\mathbb{R})$ and $h_{22}(t)<0, t\in[0,T]$.
\end{ass}

By (\ref{moncondition}), $H_{22}(t)-H_{23}^2(t)H_{33}^{-1}(t)<0,\ t\in[0,T]$.
Then by \eqref{h23t=-h33th13t},
\begin{eqnarray}\label{H22-H33H132-le-0-H22}
H_{22}(t)-H_{33}(t)H_{13}^2(t)<0,\indent   t\in[0,T].
\end{eqnarray}
\begin{rem}\label{lambda-h22-positive}
Let
\begin{eqnarray}
\ \ \bar{H}=\begin {bmatrix}
0&0&0\\
0&h_{22}&0\\
0&0&0
\end{bmatrix}.
\end{eqnarray}
Then $\bar{H}\le0_{3\times3}$.
By Remark \ref{positive real number}, all the eigenvalues of problem \eqref{generel-function-eigen-problem} are located in $\mathbb{R}^+=[0,+\infty)$.

Such a fact can also be deduced from the following viewpoint.
From (\ref{moncondition}), we have $H_{22}(t)-H_{33}(t)H_{13}^2(t)<0, t\in[0,T]$.
Then, when $\lambda h_{22}\ge0$, there is no finite blow-up time for solution $k(\cdot,\lambda)$ to (\ref{general-riccati-function-case}). By Lemma \ref{decouple-lemma}, there is none negative eigenvalue for the eigenvalue problem (\ref{generel-function-eigen-problem}).
\end{rem}

In what follows, $H_{ij}$ ($h_{22}, resp.$) can be seen as continuous functions on $(-\infty,T]$, with $H_{ij}(t)=H_{ij}(0), (h_{22}(t)=h_{22}(0), resp.), t\in(-\infty,0]$.

As in (\ref{riccati-lemma-n-sep-1})-(\ref{riccati-lemma-n-sep-3}), corresponding to (\ref{generel-function-eigen-problem}), we introduce the following Riccati equation:
\begin{equation} \label{riccati-lemma-n-sep-h22-k}
\left\{
\begin{aligned}
&-\frac{\mathrm{d}k}{\mathrm{d}t}=k\left( H_{21}+ (H_{22}-\lambda h_{22})k+ H_{23}m\right) + H_{11}+  H_{12}k+ H_{13}m, \ \ \ t\in[T_1,T_2], \\
&m=k\left( H_{31}+H_{32}k+H_{33}m\right),\indent  t\in[T_1,T_2],    \\
&k(T_2)=k_{T_2}\in S_n,
\end{aligned}
\right.
\end{equation}
and a forward SDE similar to (\ref{forward-sde-decouple-lemma}):
\begin{equation}\label{forward-sde-decouple-lemma-h22-k}
\left\{
\begin{aligned}
&\mathrm{d}x_t=[H_{21}+ \left(H_{22}-\lambda h_{22})k+ H_{23}m\right]x_t\mathrm{d}t  \\
&\indent\ \ \ + \left[ H_{31}+ H_{32}k+ H_{33}m\right]x_t\mathrm{d}B_t,    \indent   t\in [T_1,T_2],   \\
&x(T_1)=x_0.
\end{aligned}
\right.
\end{equation}

To investigate eigenvalues of (\ref{generel-function-eigen-problem}), we borrow the method in \cite{peng} to introduce a dual Hamiltonian system (about which we give a concise introduction in Appendix \ref{intro-legendre-transfor-sec}) of (\ref{generel-function-eigen-problem}):
\begin{equation} \label{dual-Hamilton-1d-h22}
\left\{
\begin{aligned}
&\mathrm{d}\tilde{x}_t=\left[\widetilde{H}_{21}\tilde{x}_t+ \widetilde{H}_{22}\tilde{y}_t+ \widetilde{H}_{23}\tilde{z}_t\right]\mathrm{d}t    + \left[\widetilde{H}_{31}\tilde{x}_t+ \widetilde{H}_{32}\tilde{y}_t+ \widetilde{H}_{33}\tilde{z}_t\right]\mathrm{d}B_t,\ \ \  t\in[0,T], \\
&-\mathrm{d}\tilde{y}_t=\left[\widetilde{H}_{11}\tilde{x}_t+\widetilde{H}_{12}\tilde{y}_t
       + \widetilde{H}_{13}\tilde{z}_t\right]\mathrm{d}t-\tilde{z}_t \mathrm{d}B_t, \indent   t\in[0,T],  \\
&\tilde{x}(0)=0,\indent   \tilde{y}(T)=0,
\end{aligned}
\right.
\end{equation}
whose coefficients matrices are as follows:
\begin{eqnarray*}
(\widetilde{H}_{ij})_{3\times3}
=\begin {bmatrix}
H_{23}H_{33}^{-1}H_{32}-H_{22}+\lambda h_{22}&H_{23}H_{33}^{-1}H_{31}-H_{21}&- H_{23}H_{33}^{-1}\\
H_{13}H_{33}^{-1}H_{32}-H_{12}&H_{13}H_{33}^{-1}H_{31}-H_{11}&- H_{13}H_{33}^{-1}\\
- H_{33}^{-1}H_{32}&- H_{33}^{-1}H_{31}&H_{33}^{-1}
\end{bmatrix}
\end{eqnarray*}
\begin{eqnarray}
\ \ {=}\begin {bmatrix}
H_{13}^2H_{33}-H_{22}+\lambda h_{22}&-H_{13}^2-H_{21}&H_{13}\\
-H_{13}^2-H_{21}&H_{13}^2H_{33}^{-1}-H_{11}&- H_{13}H_{33}^{-1}\\
H_{13}&- H_{33}^{-1}H_{31}&H_{33}^{-1}
\end{bmatrix},\label{tilde-H-h22-1d-simp}
\end{eqnarray}
where the equality in  \eqref{tilde-H-h22-1d-simp} is from \eqref{h23t=-h33th13t}.

Note that the boundary condition in Legendre dual transformation \eqref{dual-Hamilton-1d-h22} is degenerate.

Corresponding to (\ref{dual-Hamilton-1d-h22}), similar to  (\ref{riccati-lemma-n-sep-1})-(\ref{riccati-lemma-n-sep-3}) again, there is a dual Riccati equation:
\begin{equation}  \label{riccati-lemma-1-sep-h22}
\left\{
\begin{aligned}
&-\frac{\mathrm{d}\tilde{k}}{\mathrm{d}t}=\tilde{k}\left( \widetilde{H}_{21}+ \widetilde{H}_{22}\tilde{k}+ \widetilde{H}_{23}\tilde{m}\right) + \widetilde{H}_{11}+  \widetilde{H}_{12}\tilde{k}+ \widetilde{H}_{13}\tilde{m}, \indent   t\in [T_1,T_2],    \\
&\tilde{m}=\tilde{k}\left( \widetilde{H}_{31}+\widetilde{H}_{32}\tilde{k}+\widetilde{H}_{33}\tilde{m}\right), \indent   t\in [T_1,T_2],\\
&\tilde{k}\left(T_2\right)=\tilde{k}_{T_2},
\end{aligned}
\right.
\end{equation}
and forward SDE in the form of (\ref{forward-sde-decouple-lemma}):
\begin{equation} \label{forward-sde-decouple-lemma-1-dim-h22}
\left\{
\begin{aligned}
&\mathrm{d}\tilde{x}_t=\left[\widetilde{H}_{21}+ \widetilde{H}_{22}\tilde{k}+ \widetilde{H}_{23}\tilde{m}\right]\tilde{x}_t\mathrm{d}t   + \left[ \widetilde{H}_{31}+ \widetilde{H}_{32}\tilde{k}+ \widetilde{H}_{33}\tilde{m}\right]\tilde{x}_t\mathrm{d}B_t, \ \ \  t\in \left[T_1,T_2\right],   \\
&\tilde{x}(T_1)=x_0.
\end{aligned}
\right.
\end{equation}

For any $\bar{t}\in[0,T]$, in \eqref{riccati-lemma-n-sep-h22-k} and \eqref{riccati-lemma-1-sep-h22},  take $T_2=\bar{t}$, $k_{T_2}=0$ and $\tilde{k}_{T_2}=0$.
Then by \eqref{h23t=-h33th13t},
(\ref{riccati-lemma-n-sep-h22-k}) can  be simplified to:
\begin{equation} \label{general-riccati-function-case}
\left\{
\begin{aligned}
&-\frac {\mathrm{d}k}{\mathrm{d}t}=\left(2H_{21}+H_{13}^2\right)k
+H_{11}+\left(H_{22}-H_{33}H_{13}^2-\lambda h_{22}\right)k^2, \indent   t\le\bar{t},  \\
&k(\bar{t})=0,
\end{aligned}
\right.
\end{equation}
and (\ref{riccati-lemma-1-sep-h22})  can be simplified to:
\begin{equation} \label{general-dual-riccati-function}
\left\{
\begin{aligned}
&-\frac{\mathrm{d}\tilde{k}}{\mathrm{d}t}
=-\left(2H_{21}+H_{13}^2\right)\tilde{k}-H_{11}\tilde{k}^2
-\left(H_{22}-H_{33}H_{13}^2-\lambda h_{22}\right),\indent   t\le\bar{t},      \\
&\tilde{k}(\bar{t})=0.
\end{aligned}
\right.
\end{equation}

\textbf{Notations}: We give the following notations about the blow-up time of solutions to Riccati equations:
$$t_{\lambda,\bar{t}}^k\triangleq\sup\left\{t_0\bigg|t_0<\bar{t}, k(\bar{t};\lambda)=0,\lim_{t\searrow t_0} k(t;\lambda)=+\infty\right\}$$
with respect to (\ref{general-riccati-function-case}),
and
$$t_{\lambda,\bar{t}}^{\tilde{k}}\triangleq \sup\left\{t_0\bigg|t_0<\bar{t}, \tilde{k}(\bar{t};\lambda)=0, \lim_{t\searrow t_0} \tilde{k}(t;\lambda)=-\infty\right\}$$
with respect to (\ref{general-dual-riccati-function}).

From Legendre transformation, $k=\tilde{k}^{-1}$ whenever both of them are not equal to $0$,
then
$$t_{\lambda,\bar{t}}^k=\sup\left\{t_0\bigg|t_0<\bar{t},k(\bar{t};\lambda)=0, \lim_{t\searrow t_0} \tilde{k}(t;\lambda)=0\right\}$$
and
$$t_{\lambda,\bar{t}}^{\tilde{k}}=\sup\left\{t_0\bigg|t_0<\bar{t},\tilde{k}(\bar{t};\lambda)=0, \lim_{t\searrow t_0} k(t;\lambda)=0\right\}.$$

\emph{To simplify the notation, we will write $t_{\rho,\bar{t}}^k$ ($t_{\rho,\bar{t}}^{\tilde{k}}, resp.$) as $t_{\rho}^k$ ($t_{\rho}^{\tilde{k}}, resp.$) whenever without confusion.}


\section{The existence of eigenvalues}\label{section-existence-eigenvalue}
Analysing the property of blow-up time of solutions of Riccati equations with respect to its parameter $\lambda$ plays an important role, which is an enhanced version of the idea in \cite{peng}. However, the proofs are more complicated due to the time-dependent coefficients.

\begin{lem}  \label{lemma-k-function-h22}
The blow-up time $t_\lambda^k$ of solution $k(\cdot;\lambda)$ to (\ref{general-riccati-function-case}) is
increasing with respect to $\lambda$.
Besides,
\begin{equation} \label{limit-t-lambda-k-h22-1}
\lim_{\lambda\nearrow +\infty}t_\lambda^k=\bar{t}.
\end{equation}
\end{lem}
\begin{proof}
We firstly prove (\ref{limit-t-lambda-k-h22-1}). By \eqref{moncond-H11H22H33} and \eqref{moncond-H22H23H33-1H32}, there is a $\beta>0$, such that for $t\in[0,T]$,
$$H_{11}(t)\ge \beta,\quad H_{22}(t)\le -\beta,\quad H_{33}(t)\le -\beta,\ \text{ and }\ \ H_{22}(t)-H_{33}(t)H_{13}^2(t)<0.$$
By Assumption \ref{assumption-1d-general-perturbation}, $h_{22}$ is continuous and $h_{22}(t)<0,\ \forall t\in[0,T]$.
Then we can choose two constants $\check{h}_{22}, \hat{h}_{22}$, such that
\begin{eqnarray*}
\check{h}_{22}\le h_{22}(t)\le \hat{h}_{22}<0,\indent  \forall t\in[0,T].
\end{eqnarray*}
Denote by $k_1(\cdot;\lambda)$ the solution to the following equation:
\begin{equation} \label{general-riccati-function-case-h22-k1}
\left\{
\begin{aligned}
&-\frac {\mathrm{d}k_1}{\mathrm{d}t}=\left(2H_{21}+H_{13}^2\right)k_1
+\beta+\left(H_{22}-H_{33}H_{13}^2-\lambda \hat{h}_{22}\right)k_1^2, \indent  t\le\bar{t},  \\
&k_1(\bar{t})=0.
\end{aligned}
\right.
\end{equation}
Since $\beta>0$, by Lemma \ref{const-term-posit-equation-posit}, for $t\le\bar{t}$, $k_1(t;\lambda)\ge0$.
Since $H_{11}(t)\ge \beta$ and $-\lambda h_{22}\ge-\lambda \hat{h}_{22}, t\in[0,T]$,
by Lemma \ref{elementary-compare}, $k(t;\lambda)\ge k_1(t;\lambda)\ge0, t\le\bar{t}$,
whence $t_\lambda^{k_1}\le t_\lambda^{k}<\bar{t}$.
Moreover, denote by $k_2(\cdot;\lambda)$ the solution to the following equation:
\begin{equation}  \label{general-riccati-function-case-h22-k2}
\left\{
\begin{aligned}
&-\frac {\mathrm{d}k_2}{\mathrm{d}t}=\frac{1}{2}\beta - \frac{1}{2}\lambda\hat{h}_{22}k_2^2, \indent  t\le\bar{t},  \\
&k_2(\bar{t})=0.
\end{aligned}
\right.
\end{equation}
Subtracting (\ref{general-riccati-function-case-h22-k2}) from (\ref{general-riccati-function-case-h22-k1}):
\begin{equation*} 
\left\{
\begin{aligned}
&-\frac {\mathrm{d}(k_1-k_2)}{\mathrm{d}t}=-\frac{1}{2}\lambda\hat{h}_{22}(k_1+k_2)(k_1-k_2)+\frac{\beta}{4}         \\
&\indent\indent\indent\indent\ \ +\left(2H_{21}(t)+H_{13}^2(t)\right)k_1
+\frac{\beta}{4}-\frac{\lambda}{4}\hat{h}_{22}k_1^2\\
&\indent\indent\indent\indent\ \ +\left(H_{22}(t)-H_{33}(t)H_{13}^2(t)
-\frac{\lambda}{4} \hat{h}_{22}\right)k_1^2, \indent t\le\bar{t},  \\
&(k_1-k_2)(\bar{t})=0.
\end{aligned}
\right.
\end{equation*}
\noindent
Since $H_{22}(t)-H_{33}(t)H_{13}^2(t)$ is bounded, for sufficiently large $\lambda$,
$$H_{22}(t)-H_{33}(t)H_{13}^2(t)-\frac{\lambda}{4} \hat{h}_{22} \ge0, \indent   \forall t\in[0,T].$$
Moreover, for sufficiently large $\lambda$ and $t\le\bar{t}$,
{\small\begin{equation*}
\begin{aligned}
&\left(2H_{21}(t)+H_{13}^2(t)\right)k_1\left(t;\lambda\right)+\frac{\beta}{4}
-\frac{\lambda}{4}\hat{h}_{22}k_1^2\left(t;\lambda\right)  \\
&\ \ \ \ \ = \left( \frac{ \sqrt{ -\lambda \hat{h}_{22}} k_1}{2} + \left(2H_{21}(t)+H_{13}^2(t)\right) \frac{1}{\sqrt{ -\lambda \hat{h}_{22}}}  \right)^2
+\frac{\beta}{4} +  \frac{\left(2H_{21}(t)+H_{13}^2(t)\right)^2}{\lambda \hat{h}_{22}}
\ge0.
\end{aligned}
\end{equation*}}
Then for
$t\le\bar{t}$ and sufficiently large $\lambda$,
\begin{eqnarray*}
&& \frac{1}{4}\beta
 +\left(\left(2H_{21}(t)+H_{13}^2(t)\right)k_1\left(t;\lambda\right)
 +\frac{\beta}{4}-\frac{\lambda}{4}\hat{h}_{22}k_1^2\left(t;\lambda\right)\right) \\
&&\indent  +\left(H_{22}(t)-H_{33}(t)H_{13}^2(t)-\frac{\lambda}{4} \hat{h}_{22}\right)k_1^2\left(t;\lambda\right)\ge\frac{\beta}{4}>0.
\end{eqnarray*}
By Lemma \ref{elementary-compare}, $k_1(t;\lambda)\ge k_2(t;\lambda)\ge0, t\le\bar{t}$,
whence $t_\lambda^{k_2}\le t_\lambda^{k_1}$.
It follows that, for sufficiently large $\lambda$,
$$t_\lambda^{k_2}\le t_\lambda^{k_1}\le t_\lambda^{k}<\bar{t}.$$
Moreover, (\ref{general-riccati-function-case-h22-k2}) is an equation with constant coefficients, and
\begin{equation*}
  k_2(t)=\sqrt{\frac{\beta}{\lambda (-\hat{h}_{22})}} \tan\left[ \frac{\sqrt{\lambda\beta(-\hat{h}_{22})}}{2} \left(\bar{t}-t\right)\right], \indent  \lambda>0,\ t\le\bar{t}.
\end{equation*}
Then
\begin{equation*}
  \frac{\sqrt{\lambda\beta(-\hat{h}_{22})}}{2} \left(\bar{t}-t_{\lambda}^{k_2}\right)=\frac{\pi}{2},
\end{equation*}
from which we get
$$\lim_{\lambda\nearrow +\infty} t_\lambda^{k_2}=\bar{t}.$$
Then $\lim_{\lambda\nearrow +\infty}t_\lambda^k=\bar{t}$, which is (\ref{limit-t-lambda-k-h22-1}).

For any $\lambda_1>\lambda_2>0$,
since $h_{22}(t)<0$,
$$-\lambda_1 h_{22}(t)>-\lambda_2 h_{22}(t),\indent   t\in[0,T].$$
By Lemma \ref{comparison theorem} and Lemma \ref{const-term-posit-equation-posit},
$$k(t;\lambda_1)\ge k(t;\lambda_2)\ge0,\indent   t\le \bar{t}.$$
Therefore,  for $0<\lambda_2<\lambda_1$, $t_{\lambda_2}^k\le t_{\lambda_1}^k$.
\end{proof}

Since $t_\lambda^k$ is increasing in $\lambda$ and $\lim_{\lambda\nearrow+\infty}t_\lambda^{k}=\bar{t}$,
we can define
\begin{equation*} 
\lambda_0\left(\bar{t},k\right)\triangleq \inf\left\{\lambda\Big|\lambda\ge0, t_\lambda^k>-\infty, k(\bar{t};\lambda)=0\right\}.
\end{equation*}
Set
\begin{eqnarray}\label{lambda-b-h22-uniform-posit}
&&\lambda_b\triangleq \frac{\min_{t\in[0,T]}\left\{H_{22}(t)-H_{33}(t)H_{13}^2(t)\right\}}
{\max_{t\in[0,T]}\left\{h_{22}(t)\right\}} .
\end{eqnarray}
Then obviously
\begin{equation}\label{lambda-b-h22-uniform-posit-quadratic-posit}
H_{22}(t)-H_{33}(t)H_{13}^2(t)-\lambda h_{22}(t)>0,\indent   t\in[0,T],\   \lambda\ge\lambda_b.
\end{equation}

\begin{lem}  \label{lemma-k-function-h22-copy}
Following the notations above, the blow-up time $t_\lambda^k$
is continuous and strictly increasing in
$\left(\lambda_0(\bar{t},k)\vee \lambda_b,+\infty\right)$.
\end{lem}

\begin{proof}

Firstly, we prove that $t_\lambda^k$ is continuous in $\left(\lambda_0(\bar{t},k)\vee \lambda_b,+\infty\right)$.
Recall that for $\lambda'\in(\lambda_0(\bar{t},k)\vee \lambda_b, +\infty)$, the blow-up time $t_{\lambda'}^k$   satisfies
$\lim_{t\searrow t_{\lambda'}^k}k\left(t;\lambda'\right)=+\infty$.
Then in (\ref{general-riccati-function-case}), there is a $\delta_1>0$, such that $$H_{22}(t_{\lambda'}^k)-H_{33}(t_{\lambda'}^k)H_{13}^2(t_{\lambda'}^k)-\lambda h_{22}(t_{\lambda'}^k)>\delta_1>0.$$
Moreover, in (\ref{general-dual-riccati-function}),
since $\tilde{k}(t_{\lambda'}^k; \lambda')=0$,
$$-\frac{\mathrm{d}\tilde{k}}{\mathrm{d}t}\bigg|_{t=t_{\lambda'}^k}
=-\left(H_{22}(t)-H_{33}(t)H_{13}^2(t)-\lambda h_{22}(t)\right)\bigg|_{t=t_{\lambda'}^k}<-\delta_1.$$
Further, by the continuity of $-\frac{\mathrm{d}\tilde{k}}{\mathrm{d}t}$ and $\tilde{k}\left(t_{\lambda'}^k; \lambda'\right)=0$,
there is a $\delta_2>0$, such that
\begin{eqnarray*}
-\frac{\mathrm{d}\tilde{k}}{\mathrm{d}t}\bigg|_{t=t_0}<-\frac{\delta_1}{2}<0, \indent
\forall t_0\in\left[t_{\lambda'}^k-\delta_2,t_{\lambda'}^k+\delta_2\right].
\end{eqnarray*}
Then for $\forall \epsilon_1\in(0,\delta_2)$, by Lagrangian Middle-Value Theorem,
$$\tilde{k}\left(t_{\lambda'}^k-\epsilon_1; \lambda'\right)<-\frac{\delta_1\epsilon_1}{2}<0.$$
By the continuous dependence of solution $\tilde{k}$ to (\ref{general-dual-riccati-function}) with respect to parameter $\lambda'$:
\begin{equation*}
\lim_{\lambda\rightarrow\lambda'} \left|\tilde{k}\left(t_{\lambda'}^k-\epsilon_1; \lambda'\right)-\tilde{k}\left(t_{\lambda'}^k-\epsilon_1; \lambda\right)\right|=0,
\end{equation*}
there is a $\delta_{\epsilon_1}>0$, such that for $\forall \lambda\in(\lambda'-\delta_{\epsilon_1},\lambda'+\delta_{\epsilon_1})$,
$\tilde{k}\left(t_{\lambda'}^k-\epsilon_1; \lambda\right)<0$.
Then by the definition of $t_{\lambda}^k$,
\begin{equation}\label{h22-t-lambda-k-continu-left}
t_{\lambda}^k>t_{\lambda'}^k-\epsilon_1.
\end{equation}
On the other hand, choose $\bar{t}_1\in(t_{\lambda'}^k, \bar{t})$. Recall that from Legendre transformation, $k^{-1}(t;\lambda)=\tilde{k}(t;\lambda)$ whenever both of them are not 0.
Besides, $\tilde{k}(t;\lambda')>0, t\in(t_{\lambda'}^k, \bar{t}_1]$.
Then we consider (\ref{general-dual-riccati-function}) with terminal condition $\tilde{k}(\bar{t}_1; \lambda')=k^{-1}(\bar{t}_1; \lambda')$:
\begin{equation} \label{tilde-k-h22-conti-exten}
\left\{
\begin{aligned}
&-\frac{\mathrm{d}\tilde{k}}{\mathrm{d}t}
=-\left(2H_{21}+H_{13}^2\right)\tilde{k}-H_{11}\tilde{k}^2
-\left(H_{22}-H_{33}H_{13}^2-\lambda' h_{22}\right),\indent   t\le\bar{t}_1,      \\
&\tilde{k}(\bar{t}_1)=k^{-1}(\bar{t}_1; \lambda').
\end{aligned}
\right.
\end{equation}
Solution $\tilde{k}(\cdot;\lambda')$ to (\ref{tilde-k-h22-conti-exten}) can be extended to $[\bar{t}_2,\bar{t}_1]\supsetneqq [t^k_{\lambda'}, \bar{t}_1]$ due to its local Lipschitz coefficients.
From the continuous dependence of solution $\tilde{k}$ to (\ref{general-dual-riccati-function}) with respect to parameter $\lambda'$,
$$\lim_{\lambda\rightarrow\lambda'} \sup_{t\in[\bar{t}_2,\bar{t}_1]}\left|\tilde{k}(t;\lambda')-\tilde{k}(t;\lambda)\right|=0.$$
For any sufficiently small $\epsilon_2>0$, $\tilde{k}(t;\lambda')$ have uniform strictly positive lower bound for $t\in[t^k_{\lambda'}+\epsilon_2, \bar{t}_1]$.
Then there is a $\delta_{\epsilon_2}>0$, such that for $\forall \lambda\in(\lambda'-\delta_{\epsilon_2},\lambda'+\delta_{\epsilon_2})$ and
$\forall t\in\left[t^k_{\lambda'}+\epsilon_2, \bar{t}_1\right]$, $\tilde{k}(t;\lambda)>0$.
Then by the definition of $t_{\lambda}^k$,
\begin{equation}\label{h22-t-lambda-k-continu-right}
t_\lambda^k<t^k_{\lambda'}+\epsilon_2.
\end{equation}
From \eqref{h22-t-lambda-k-continu-left} and \eqref{h22-t-lambda-k-continu-right},
$t_\lambda^k$ is continuous in $\lambda$.

At last,  we prove that $t_\lambda^k$ is strictly increasing with respect to $\lambda$.
For $\lambda>\lambda'$,
$$-\lambda h_{22}(t)>-\lambda' h_{22}(t),\indent   t\in[0,T].$$
Then by Lemma \ref{elementary-compare}, $k(t;\lambda)>k(t;\lambda'), t<\bar{t}$.
In particular, $k(\bar{t}_1;\lambda)>k(\bar{t}_1;\lambda')$, whence $\tilde{k}(\bar{t}_1;\lambda)<\tilde{k}(\bar{t}_1;\lambda')$.
Consider the following two ODEs with terminal time $\bar{t}_1$:
\begin{equation*} 
\left\{
\begin{aligned}
&-\frac{\mathrm{d}\tilde{k}}{\mathrm{d}t}
=-\left(2H_{21}+H_{13}^2\right)\tilde{k}-H_{11}\tilde{k}^2
-\left(H_{22}-H_{33}H_{13}^2-\lambda h_{22}\right),\indent   t\le\bar{t}_1,      \\
&\tilde{k}(\bar{t}_1)=\tilde{k}(\bar{t}_1;\lambda),
\end{aligned}
\right.
\end{equation*}
and
\begin{equation*} 
\left\{
\begin{aligned}
&-\frac{\mathrm{d}\tilde{k}}{\mathrm{d}t}
=-\left(2H_{21}+H_{13}^2\right)\tilde{k}-H_{11}\tilde{k}^2
-\left(H_{22}-H_{33}H_{13}^2-\lambda' h_{22}\right),\indent   t\le\bar{t}_1,      \\
&\tilde{k}(\bar{t}_1)=\tilde{k}(\bar{t}_1;\lambda').
\end{aligned}
\right.
\end{equation*}
From $\tilde{k}(\bar{t}_1;\lambda)<\tilde{k}(\bar{t}_1;\lambda')$ and $-\lambda h_{22}(t)>-\lambda' h_{22}(t), t\in[0,T]$, by Lemma \ref{elementary-compare}, we have $\tilde{k}(t; \lambda)< \tilde{k}(t; \lambda'), t\le \bar{t}_1$.
In particular, $\tilde{k}(t_{\lambda'}^k; \lambda)< \tilde{k}(t_{\lambda'}^k; \lambda')=0$, and hence for $\lambda>\lambda'$, $t_{\lambda}^k>t_{\lambda'}^k$.
\end{proof}

Next, we also need the similar results of $t_\lambda^{\tilde{k}}$ with respect to $\lambda$, the proofs of which are similar to that of Lemma \ref{lemma-k-function-h22} and Lemma \ref{lemma-k-function-h22-copy}. We will write down the proofs in appendix  for readers' convenience.

\begin{lem}   \label{lemma-k-tilde-function-h22}
The blow-up time $t_\lambda^{\tilde{k}}$ of solution $\tilde{k}(\cdot;\lambda)$ to (\ref{general-dual-riccati-function})
is increasing in $\lambda$,
and
\begin{equation} \label{limit-tilde-t-lambda-k-h22-1}
\lim_{\lambda\nearrow+\infty}t_\lambda^{\tilde{k}}=\bar{t}.
\end{equation}
\end{lem}

Since $t_{\lambda}^{\tilde{k}}$ is increasing and $\lim_{\lambda\nearrow+\infty}t_\lambda^{\tilde{k}}=\bar{t}$,
we can define
\begin{equation*} 
\lambda_0\left(\bar{t},\tilde{k}\right)\triangleq \inf\left\{\lambda\Big|\lambda\ge0, t_{\lambda}^{\tilde{k}}>-\infty, \tilde{k}(\bar{t};\lambda)=0\right\}.
\end{equation*}
\begin{lem}   \label{lemma-k-tilde-function-h22-cop}
Following the above notations, the blow-up time $t_\lambda^{\tilde{k}}$
is continuous and strictly increasing in $\left(\lambda_0\left(\bar{t},\tilde{k}\right)\vee \lambda_b,+\infty\right)$.
\end{lem}

To get eigenvalues of \eqref{generel-function-eigen-problem}, we also need property of $t^k_{\lambda,\bar{t}}$ and $t^{\tilde{k}}_{\lambda,\bar{t}}$ in $\bar{t}$ depicted by the following two lemmas.
\begin{lem} \label{lemma-riccati-func-mon-terminal-h22}
For $0\le\bar{t}_2<\bar{t}_1\le T$, denote separately by $k_i(\cdot;\lambda),\ i=1,2$  the solution to (\ref{riccati-func-mon-terminal-h22})
\begin{equation}\label{riccati-func-mon-terminal-h22}
\left\{
\begin{aligned}
&-\frac {\mathrm{d}k_i}{\mathrm{d}t}=\left(2H_{21}+H_{13}^2\right)k_i
+H_{11}+\left( H_{22}-H_{33}H_{13}^2 -\lambda h_{22}\right)k_i^2, \indent  t\le \bar{t}_i,\\
&k_i(\bar{t}_i)=0.
\end{aligned}
\right.
\end{equation}
Then  $t_\lambda^{k_2}\le t_\lambda^{k_1}$(if finite).
Besides, $t_\lambda^{k_i}$ is continuous dependent on $\bar{t}_i,\ i=1,2$.
\end{lem}
\begin{proof}
By (\ref{moncond-H11H22H33}), there is a $\beta>0$, such that $H_{11}>\beta>0$. Then by  Lemma \ref{const-term-posit-equation-posit}, $k_i(t;\lambda)\ge0,\  t\le \bar{t}_i$.
If $\bar{t}_2\le t_\lambda^{k_1}$, then obviously $t_\lambda^{k_2}<\bar{t}_2\le t_\lambda^{k_1}$.
If $\bar{t}_2> t_\lambda^{k_1}$, then $k_2(\bar{t}_2;\lambda)=0<k_1(\bar{t}_2;\lambda)<\infty$.
By Lemma \ref{comparison theorem}, for $t\le\bar{t}_2$,
we have
$$0\le k_2(t;\lambda)\le k_1(t;\lambda),$$
whence $t_\lambda^{k_2}\le t_\lambda^{k_1}$.
That $t_\lambda^{k_i}$ is continuous dependent on $\bar{t}_i$ is from local Lipschitz condition of (\ref{riccati-func-mon-terminal-h22}).
\end{proof}
\begin{lem} \label{lemma-riccati-dual-func-mon-terminal-h22}
For $0\le\bar{t}_2<\bar{t}_1\le T$
and $\lambda\ge\lambda_b$,  assume that $k_i(\cdot;\lambda),\ i=1,2$ is separately the solution to (\ref{riccati-dual-func-mon-terminal-h22}):
\begin{equation}  \label{riccati-dual-func-mon-terminal-h22}
\left\{
\begin{aligned}
&-\frac{\mathrm{d}\tilde{k}_i}{\mathrm{d}t}
=-(2H_{21}+H_{13}^2)\tilde{k}_i-H_{11}\tilde{k}_i^2
-(H_{22}-H_{33}H_{13}^2-\lambda h_{22}), \indent  t\le \bar{t}_i,  \\
&\tilde{k}_i(\bar{t}_i)=0.
\end{aligned}
\right.
\end{equation}
Then $t_\lambda^{k_2}\le t_\lambda^{k_1}$ (if finite).
Besides, $t_\lambda^{\tilde{k}_i}$ is continuous dependent on $\bar{t}_i,\ i=1,2$.
\end{lem}
By \eqref{lambda-b-h22-uniform-posit-quadratic-posit}, when $\lambda\ge\lambda_b$,
$H_{22}(t)-H_{33}(t)H_{13}^2(t)-\lambda h_{22}(t)>0,\ t\in[0,T]$.
The proof of Lemma \ref{lemma-riccati-dual-func-mon-terminal-h22} is similar to that of Lemma  \ref{lemma-riccati-func-mon-terminal-h22} and is omitted.

\begin{thm}  \label{general-lambda-exist}
Under Assumption \ref{assumption-1d-general-perturbation},
there exists $\{\lambda_m\}_{m=1}^\infty \subset (\lambda_b, +\infty)$, all those eigenvalues  of
problem (\ref{generel-function-eigen-problem}) contained in $(\lambda_b, +\infty)$, satisfying $\lambda_m\rightarrow+\infty$ as \ $m\rightarrow +\infty$.
Moreover, the eigenfunction space corresponding to each $\lambda_m$ is of  $1$ dimension.
\end{thm}

\begin{proof}
At first,  consider the Riccati equation \eqref{general-riccati-function-case} and take $\bar{t}=T$:
\begin{equation} \label{general-riccati-function-case-theo-1}
\left\{
\begin{aligned}
&-\frac {\mathrm{d}k}{\mathrm{d}t}=\left(2H_{21}+H_{13}^2\right)k
+H_{11}+\left(H_{22}-H_{33}H_{13}^2-\lambda h_{22}\right)k^2, \indent  t\le T,  \\
&k(T)=0,
\end{aligned}
\right.
\end{equation}
Denote by $t_1(\lambda)(<T)$ the blow-up time of solution $k(\cdot;\lambda)$ to \eqref{general-riccati-function-case-theo-1}.
Then by Lemma \ref{lemma-k-function-h22} and Lemma \ref{lemma-k-function-h22-copy},
\begin{equation*}
t_1(\cdot): \left(\lambda_0(T,k)\vee\lambda_b,+\infty\right) \rightarrow \left(\lim_{\lambda'\searrow\{\lambda_0(T,k)\vee\lambda_b\}}t^k_{\lambda',T}\ ,\ T\right)
\end{equation*}
is a strictly increasing and continuous bijective mapping of $\lambda$.

Then consider the dual Riccati equation \eqref{general-dual-riccati-function} and take $\bar{t}=t_1(\lambda)$:
\begin{equation} \label{general-dual-riccati-function-theo-2}
\left\{
\begin{aligned}
&-\frac{\mathrm{d}\tilde{k}}{\mathrm{d}t}
=-\left(2H_{21}+H_{13}^2\right)\tilde{k}-H_{11}\tilde{k}^2
-\left(H_{22}-H_{33}H_{13}^2-\lambda h_{22}\right),\indent  t\le t_1(\lambda),      \\
&\tilde{k}(t_1(\lambda))=0.
\end{aligned}
\right.
\end{equation}
Let $t_2(\lambda)$ be the blow-up time of the solution $\tilde{k}(\cdot;\lambda)$ to \eqref{general-dual-riccati-function-theo-2}.
By Lemmas  \ref{lemma-k-function-h22} to  \ref{lemma-riccati-dual-func-mon-terminal-h22},
\begin{eqnarray*}
\lim_{\lambda\rightarrow +\infty} t_2(\lambda) = T,
\end{eqnarray*}
and $t_2(\cdot)$ is a strictly increasing and continuous bijective mapping  once it is finite for sufficiently large $\lambda$.

Then consider the Riccati equation \eqref{general-riccati-function-case} and take $\bar{t}=t_2(\lambda)$:
\begin{equation} \label{general-riccati-function-case-theo-3}
\left\{
\begin{aligned}
&-\frac {\mathrm{d}k}{\mathrm{d}t}=\left(2H_{21}+H_{13}^2\right)k
+H_{11}+\left(H_{22}-H_{33}H_{13}^2-\lambda h_{22}\right)k^2, \indent t\le t_2(\lambda),  \\
&k(t_2(\lambda))=0,
\end{aligned}
\right.
\end{equation}
Let $t_3(\lambda)(<t_2(\lambda))$ be the blow-up time of solution $k(\cdot;\lambda)$ to \eqref{general-riccati-function-case-theo-3}.
Then by Lemmas \ref{lemma-k-function-h22} to \ref{lemma-riccati-dual-func-mon-terminal-h22},
\begin{eqnarray*}
\lim_{\lambda\rightarrow +\infty} t_3(\lambda) = T,
\end{eqnarray*}
and $t_3(\cdot)$ is a  strictly increasing continuous bijective mapping once it is finite for sufficiently large $\lambda$.

By induction, we can define $t_m(\cdot),\ m=1,2,3,\cdots$ as above, such that for sufficiently large $\lambda$,
$$\cdots<t_3(\lambda)<t_2(\lambda)<t_1(\lambda)<t_0(\lambda)\triangleq T.$$
Since for any fixed $\lambda'>\lambda_b$,
$\inf\left\{ \bar{t}-t_{\lambda',\bar{t}}^k \Big|\bar{t}\in[0,T]\right\}\wedge \inf\left\{ \bar{t}-t_{\lambda',\bar{t}}^{\tilde{k}} \Big|\bar{t}\in[0,T]\right\} >0$,
then there is $n\in\mathbb{N}_+\cup\{0\}$ and $\lambda\ge\lambda_b$, such that
\begin{equation*}
\begin{aligned}
T-t_{1+2n}(\lambda) =& \sum_{i=1}^{1+2n} \left[ t_{i-1}(\lambda)-t_{i}(\lambda) \right]\\
=& \sum_{i=0}^{n}\left[t_{2i}(\lambda)- t^k_{\lambda,t_{2i}(\lambda)} \right]
    +\sum_{i=1}^{n} \left[ t_{2i-1}(\lambda)-t^{\tilde{k}}_{\lambda,t_{2i-1}(\lambda)} \right] \\
>&T.
\end{aligned}
\end{equation*}
It deduces that for the above $n$ and $\lambda$,
\begin{equation*}
t_{1+2n}(\lambda)<0.
\end{equation*}
Further, because $t_{1+2n}(\cdot)$ is a strictly increasing continuous bijective mapping and
$$\lim_{\lambda\nearrow+\infty} t_{1+2n}(\lambda)=T,$$
there is a unique minimal $\lambda_1>\lambda_b$ and certain unique  minimal $2n\in\mathbb{N}_+\cup\{0\}$,
such that
\begin{equation}\label{the-n-exist-t2npl1=0}
t_{1+2n}(\lambda_1)=0.
\end{equation}
Moreover, by Lemmata \ref{lemma-k-function-h22}-\ref{lemma-riccati-dual-func-mon-terminal-h22} again, those functions
\begin{equation}\label{those-function-tlambda}
t_{2m+1+2n}(\lambda),\quad   m=0,1,2,\cdots
\end{equation}
are also strictly increasing continuous bijective mapping and
\begin{equation*}
\lim_{\lambda\nearrow+\infty} t_{2m+1+2n}(\lambda)=T, \quad   m=0,1,2,\cdots.
\end{equation*}
Then there is a unique $\lambda_{m+1}\in(\lambda_m,+\infty)$, such that $t_{2m+1+2n}(\lambda_{m+1})=0,\ m=0,1,2,\cdots$.
This derives a series of $\lambda_m,\ m=0,1,2,\cdots$, satisfying $\lambda_b<\lambda_1<\lambda_2<\lambda_3<\cdots$ and $t_{2m+1+2n}(\lambda_{m+1})=0$.

We claim that this series of $\lambda_m,\ m=1,2,\cdots$, are exactly all the eigenvalues of problem (\ref{generel-function-eigen-problem}) which are contained in $(\lambda_b, +\infty)$.

To prove the claim, for $\lambda_m$, $m=1,2,3,\cdots$, we construct the associated  eigenfunctions.
By the above procedure,
\begin{equation*}
0=t_{2m-1+2n}(\lambda_m)<t_{2m-2+2n}(\lambda_{m})<\cdots<t_2(\lambda_m)<t_1(\lambda_m)<T.
\end{equation*}
Divide the interval $[0,T]$ into $2m+2n$ parts:
\begin{eqnarray*}
I_1&=&\left[0,\frac{t_{2m-2+2n}(\lambda_{m})}{2}\right],\\
I_2&=&\left[\frac{t_{2m-2+2n}(\lambda_{m})}{2}, \frac{t_{2m-2+2n}(\lambda_{m})+t_{2m-3+2n}(\lambda_{m})}{2}\right],\\
 &\vdots&\\
I_{2m-1+2n}&=&\left[\frac{t_2(\lambda_m)+t_1(\lambda_m)}{2}, \frac{t_1(\lambda_m)+T}{2}\right],\\
I_{2m+2n}&=&\left[\frac{t_1(\lambda_m)+T}{2},T\right].
\end{eqnarray*}
Recall that by Legendre transformation, $k(\cdot;\lambda)\tilde{k}(\cdot;\lambda)=1$ whenever both of them are nonzero.
By the above procedure, $\tilde{k}(\cdot;\lambda_m)$ exists on $I_1\cup I_3\cup\cdots \cup I_{2m-1+2n}$,
while $k(\cdot;\lambda_m)$ exists on $I_2\cup I_4\cup\cdots \cup I_{2m+2n}$, and $\tilde{k}(0;\lambda_m)=k(T;\lambda_m)=0$.
Next, we will use Lemma \ref{decouple-lemma} to get the eigenfunctions.
Take $\tilde{x}_0\neq0$ and solve (\ref{forward-sde-decouple-lemma-1-dim-h22}) on $I_1$ with initial value condition $\tilde{x}(0)=\tilde{x}_0$:
\begin{equation} \label{forward-sde-decouple-lemma-1-dim-h22-theo-1}
\left\{
\begin{aligned}
&\mathrm{d}\tilde{x}_t=\left[\widetilde{H}_{21}+ \widetilde{H}_{22}\tilde{k}+ \widetilde{H}_{23}\tilde{m}\right]\tilde{x}_t\mathrm{d}t     + \left[ \widetilde{H}_{31}+ \widetilde{H}_{32}\tilde{k}+ \widetilde{H}_{33}\tilde{m}\right]\tilde{x}_t\mathrm{d}B_t, \indent   t\in I_1,   \\
&\tilde{x}(0)=\tilde{x}_0.
\end{aligned}
\right.
\end{equation}
By (\ref{tilde-H-h22-1d-simp}),
$$\widetilde{H}_{23}=-\widetilde{H}_{33}\widetilde{H}_{13}\indent \text{and}\indent \widetilde{H}_{13}=H_{13}.$$
Let
$$c=\frac{1}{2(\max\{|H_{13}(t)|,\ t\in [0,T]\}+1)^2}.$$
Then
$$\left(1-k(t)H_{33}(t)\right)^2 \ge  c H_{13}^2(t)\left(1-k(t)H_{33}(t)\right)^2$$
and
$$\left(1-\tilde{k}(t)\widetilde{H}_{33}(t)\right)^2 \ge  c \widetilde{H}_{13}^2(t)\left(1-\tilde{k}(t)\widetilde{H}_{33}(t)\right)^2.$$
It follows that condition (\ref{weak-unique-condition}) holds true.
By Lemma \ref{decouple-lemma}, $\left(\tilde{x},\tilde{y},\tilde{z}\right)$ uniquely exist on $I_1$ and  $\tilde{y}=\tilde{k}\tilde{x}$.
In particular, we get $\tilde{y}\left(\frac{t_{2m-2+2n}(\lambda_{m})}{2}\right)$
and $x(0)=\tilde{y}(0)=\tilde{k}(0)\tilde{x}(0)=0$.

Similarly,  we can solve (\ref{forward-sde-decouple-lemma-h22-k}) on $I_2$ with initial value condition: 
\begin{equation*} 
\left\{
\begin{aligned}
&\mathrm{d}x_t=\left[H_{21}+ \left(H_{22}-\lambda h_{22}\right)k+ H_{23}m\right]x_t\mathrm{d}t   + \left[ H_{31}+ H_{32}k+ H_{33}m\right]x_t\mathrm{d}B_t,    \ \ \  t\in I_2,   \\
&x\left(\frac{t_{2m-2+2n}(\lambda_{m})}{2}\right)=\tilde{y}\left(\frac{t_{2m-2+2n}(\lambda_{m})}{2}\right).  \end{aligned}
\right.
\end{equation*}
By Lemma \ref{decouple-lemma}, $(x,y,z)$ uniquely exist on $I_2$  and $y=kx$.

In particular, we get $y\left(\frac{t_{2m-2+2n}(\lambda_{m})+t_{2m-3+2n}(\lambda_{m})}{2}\right)$.
Then we consider (\ref{forward-sde-decouple-lemma-1-dim-h22}) on $I_3$ with initial value condition:
\begin{equation*} 
\left\{
\begin{aligned}
&\mathrm{d}\tilde{x}_t=\left[\widetilde{H}_{21}+ \widetilde{H}_{22}\tilde{k}+ \widetilde{H}_{23}\tilde{m}\right]\tilde{x}_t\mathrm{d}t    + \left[ \widetilde{H}_{31}+ \widetilde{H}_{32}\tilde{k}+ \widetilde{H}_{33}\tilde{m}\right]\tilde{x}_t\mathrm{d}B_t, \indent  t\in I_3,   \\
&\tilde{x}\left(\frac{t_{2m-2+2n}(\lambda_{m})+t_{2m-3+2n}(\lambda_{m})}{2}\right)
=y\left(\frac{t_{2m-2+2n}(\lambda_{m})+t_{2m-3+2n}(\lambda_{m})}{2}\right).
\end{aligned}
\right.
\end{equation*}
By Lemma \ref{decouple-lemma} again, $\left(\tilde{x},\tilde{y},\tilde{z}\right)$ uniquely exist on $I_3$ and  $\tilde{y}=\tilde{k}\tilde{x}$.
By induction,  we can solve (\ref{forward-sde-decouple-lemma-h22-k}) on $I_{2m+2n}$ with initial value condition:
\begin{equation*} 
\left\{
\begin{aligned}
&\mathrm{d}x_t=\left[H_{21}+ \left(H_{22}-\lambda h_{22}\right)k+ H_{23}m\right]x_t\mathrm{d}t   \\
&\indent\ \ \ + \left[ H_{31}+ H_{32}k+ H_{33}m\right]x_t\mathrm{d}B_t,    \indent   t\in I_{2m+2n},   \\
&x\left(\frac{t_1(\lambda_m)+T}{2}\right)=\tilde{y}\left(\frac{t_1(\lambda_m)+T}{2}\right),
\end{aligned}
\right.
\end{equation*}
where $\tilde{y}\left(\frac{t_1(\lambda_m)+T}{2}\right)$ is obtained from the previous step on $I_{2m-1+2n}$.
By Lemma \ref{decouple-lemma}, $(x,y,z)$ uniquely exists on $I_{2m+2n}$  and $y=kx$.
In particular, $y(T)=k(T)x(T)=0$.

Up to now, we get the unique $(x,y,z)$ on $I_2\cup I_4\cup\cdots \cup I_{2m+2n}$, and the unique $\left(\tilde{x},\tilde{y},\tilde{z}\right)$ on $I_1\cup I_3\cup\cdots \cup I_{2m-1+2n}$.
By Legendre dual transformation and Lemma \ref{decouple-lemma}, the triple $(x_t,y_t,z_t), t\in [0,T]$ defined by
\begin{equation}\label{xyz-form-dual}
 {(x_t,y_t,z_t)=}
\left\{
\begin{aligned}
  \left(\widetilde{y},\widetilde{x},\widetilde{H}_{31}\widetilde{x}+ \widetilde{H}_{32}\widetilde{y}+ \widetilde{H}_{33}\widetilde{z}\right)(t), &\indent   t\in I_1, I_3,\cdots, I_{2m-1+2n},  \\
  (x,y,z)(t), &\indent   t\in I_2,  I_4,\cdots, I_{2m+2n},
\end{aligned}
\right.
\end{equation}
is exactly the nontrivial solution to (\ref{generel-function-eigen-problem}) corresponding to eigenvalue $\lambda_m$.

Next, we will show that the space of eigenfunctions associated with each $\lambda_m$ are $1$-dimensional. By Lemma \ref{decouple-lemma}, every non-trivial solution $(x,y,z)$ to (\ref{generel-function-eigen-problem}) satisfies $\tilde{k}(0)y_0=0$.
In \eqref{forward-sde-decouple-lemma-1-dim-h22-theo-1}, taking
$$\tilde{x}'_0=\mu\tilde{x}_0=\mu y_0,\indent  \mu\in\mathbb{R}\verb"\"\{0\},$$
we get the unique solution $(x',y',z')$  by the above procedure.
On the other hand, $(\mu x,\mu y, \mu z)$ is a nontrivial solution to (\ref{generel-function-eigen-problem}) satisfying $\tilde{k}(0)\mu y_0=0$.
Moreover, by the uniqueness in Lemma \ref{decouple-lemma},
$$(x',y',z')=(\mu x,\mu y, \mu z).$$
That is to say,
the dimension of eigenfunction space corresponding to each $\lambda_m$ is $1$.

At last, we interpret why there are not other eigenvalues.
For any $\lambda>\lambda_b$, $\lambda\not=\lambda_m,\forall m\ge1$, by the above procedure, $(k,m)$ and $(\tilde{k},\tilde{m})$ exist in turn on the whole $[0,T]$.
Then by Lemma \ref{decouple-lemma}, corresponding to $\lambda>\lambda_b$, $\lambda\not=\lambda_m,\forall m\ge1$, the eigenvalue problem (\ref{generel-function-eigen-problem}) has unique  solution.
On the other hand,
for any $\lambda>\lambda_b$, $\lambda\not=\lambda_m,\forall m\ge1$, by the above procedure, $\tilde{k}(0;\lambda)\neq0$.
Then we have $y(0)=0$ from the following equality:
$$x(0)=\tilde{y}(0)=\tilde{k}(0;\lambda)\tilde{x}(0)=\tilde{k}(0;\lambda)y(0) =0.$$
Then trivial solution is the unique one to (\ref{generel-function-eigen-problem}) corresponding to any $\lambda>\lambda_b$, $\lambda\not=\lambda_m,\forall m\ge1$.
\end{proof}

\begin{rem}
In Theorem \ref{general-lambda-exist}, we can say nothing about those eigenvalues in $(0,\lambda_b]$,
However, in Section \ref{examp-find-all-eigen-sec-ass}, under some proper additional conditions, we can discover all the eigenvalues of problem \eqref{generel-function-eigen-problem}.
\end{rem}

\section{A sufficient condition to find out all the eigenvalues}\label{examp-find-all-eigen-sec-ass}
In Theorem \ref{general-lambda-exist}, under some proper conditions, all the eigenvalues of problem \eqref{generel-function-eigen-problem} located in $(\lambda_b, +\infty)$ are discovered.
On the other hand,  the example in Section \ref{exa-new-case-dual-to0} indicates that how subtle cases can be when the coefficients  are time-dependent and that it is a tough problem to find out all the eigenvalues.
However, in this section, we will show that under some sharper conditions, actually all the eigenvalues in $\mathbb{R}$ can also be discovered in time-dependent eigenvalue problem of stochastic Hamiltonian system with boundary conditions.

\begin{ass}\label{find-all-eigen-examp-ass-suf-intpr-ass}
Apart from Assumption \ref{assumption-1d-general-perturbation}, assume that
\begin{equation}\label{find-all-eigen-examp-ass-suf}
4\|H_{11}\|_{\infty} \left\|H_{22}-H_{33}H_{13}^2-\lambda_b h_{22}\right\|_{\infty}
\le \left\|2H_{21}+H_{13}^2\right\|^2_{\infty}<\frac{4}{T^2}.
\end{equation}
\end{ass}

After taking $\bar{t}=T$,  \eqref{general-riccati-function-case} becomes
\begin{equation} \label{general-riccati-function-case-examp-all-dis-compa-upperbound-equa}
\left\{
\begin{aligned}
&-\frac {\mathrm{d}k}{\mathrm{d}t}=\left(2H_{21}+H_{13}^2\right)k
+H_{11}+\left(H_{22}-H_{33}H_{13}^2-\lambda h_{22}\right)k^2, \indent t\le T,  \\
&k(T)=0,
\end{aligned}
\right.
\end{equation}
The following \eqref{general-riccati-function-case-examp-all-dis-ass1-compa-upperbound-equa} is also considered:
\begin{equation} \label{general-riccati-function-case-examp-all-dis-ass1-compa-upperbound-equa}
\left\{
\begin{aligned}
&-\frac {\mathrm{d}k_1}{\mathrm{d}t}=\left\|2H_{21}+H_{13}^2\right\|_{\infty}k_1
+\|H_{11}\|_{\infty}+\left\|H_{22}-H_{33}H_{13}^2-\lambda h_{22}\right\|_{\infty}k_1^2, \ \ t\le T,  \\
&k_1(T)=0, \indent \lambda\ge\lambda_b,
\end{aligned}
\right.
\end{equation}
From \eqref{moncond-H11H22H33}, $H_{11}(t)\ge \beta>0, t\in[0,T]$ and then $\|H_{11}\|_{\infty}>0$.
By Lemma \ref{const-term-posit-equation-posit}, $k(t;\lambda)\ge0,\ k_1(t;\lambda)\ge0,\ t\le T$.

By subtracting \eqref{general-riccati-function-case-examp-all-dis-ass1-compa-upperbound-equa} from \eqref{general-riccati-function-case-examp-all-dis-compa-upperbound-equa}, we have
\begin{equation*}
\left\{
\begin{aligned}
&-\frac {\mathrm{d}(k-k_1)}{\mathrm{d}t}=
\left(2H_{21}+H_{13}^2\right)(k-k_1)+\left(H_{22}-H_{33}H_{13}^2-\lambda h_{22}\right)(k+k_1)(k-k_1)\\
&\indent \indent\indent \indent\   +H_{11}-\|H_{11}\|_{\infty} +\left[\left(2H_{21}+H_{13}^2\right)-\left\|2H_{21}+H_{13}^2\right\|_{\infty}\right]k_1  \\
&\indent\indent \indent \indent\  +\left[ \left(H_{22}-H_{33}H_{13}^2-\lambda h_{22}\right)- \left\|H_{22}-H_{33}H_{13}^2-\lambda h_{22}\right\|_{\infty}\right]k_1^2, \indent t\le T,  \\
&(k-k_1)(T)=0.
\end{aligned}
\right.
\end{equation*}
Since for $t\le T$, $\lambda\ge\lambda_b$,
\begin{equation*}
\begin{aligned}
&H_{11}-\|H_{11}\|_{\infty} \le0,\\
&\left[\left(2H_{21}+H_{13}^2\right)-\left\|2H_{21}+H_{13}^2\right\|_{\infty}\right]k_1\le0,\\
&\left[ \left(H_{22}-H_{33}H_{13}^2-\lambda h_{22}\right)- \left\|H_{22}-H_{33}H_{13}^2-\lambda h_{22}\right\|_{\infty}\right]k_1^2\le0,
\end{aligned}
\end{equation*}
by Lemma \ref{elementary-compare},
\begin{equation*}
  0\le k(t;\lambda_b)\le k_1(t;\lambda_b) \le k_1(t;\lambda), \indent \lambda\ge\lambda_b,\ t\le T.
\end{equation*}
Then
\begin{equation}\label{find-all-eigen-examp-ass-suf-intpr-ass-lemm-first-amonglemma4}
  t_{\lambda_b,T}^k \le t_{\lambda_b,T}^{k_1} \le t_{\lambda,T}^{k_1}, \indent \lambda\ge\lambda_b.
\end{equation}
Moreover,
\begin{lem}\label{find-all-eigen-examp-ass-suf-intpr-ass-lemm-first}
Under Assumption \ref{find-all-eigen-examp-ass-suf-intpr-ass},
\begin{equation*}
 \lim_{\lambda\searrow\lambda_b} t_{\lambda,T}^{k_1}<0,
\end{equation*}
hence
\begin{equation}\label{find-all-eigen-examp-ass-suf-intpr-ass-lemm-first-amonglemma2}
\lim_{\lambda\searrow \lambda_b } t_{\lambda,T}^{k}=
 \lim_{\lambda\searrow\left\{\lambda_b\vee\lambda_0(T,k)\right\}} t_{\lambda,T}^{k}<0.
\end{equation}
\end{lem}

\begin{proof}
For $\lambda'_1\ge\lambda'_2\ge\lambda_b$,
$$\left\|H_{22}-H_{33}H_{13}^2-\lambda'_1 h_{22}\right\|_{\infty}
\ge \left\|H_{22}-H_{33}H_{13}^2-\lambda'_2 h_{22}\right\|_{\infty}.$$
Then by Lemma \ref{comparison theorem},
$$k_1(t;\lambda'_1)\ge k_1(t;\lambda'_2)\ge0,\indent t\le T, $$
and then
\begin{equation}\label{find-all-eigen-examp-ass-suf-intpr-ass-lemm-first-amonglemma3}
t_{\lambda'_1,T}^{k_1}\ge t_{\lambda'_2,T}^{k_1}.
\end{equation}

Under Assumption \ref{find-all-eigen-examp-ass-suf-intpr-ass},
$$4\|H_{11}\|_{\infty} \left\|H_{22}-H_{33}H_{13}^2-\lambda_b h_{22}\right\|_{\infty}-\left\|2H_{21}+H_{13}^2\right\|^2_{\infty}\le0.$$
Since for $\lambda\ge\lambda_b$, $\left\|H_{22}-H_{33}H_{13}^2-\lambda_b h_{22}\right\|_{\infty}$ is increasing in $\lambda$,
there is certain sufficiently large $\lambda_{b1}\ge\lambda_b$, such that
$$4\|H_{11}\|_{\infty} \left\|H_{22}-H_{33}H_{13}^2-\lambda_{b1} h_{22}\right\|_{\infty}-\left\|2H_{21}+H_{13}^2\right\|^2_{\infty}=0,$$
and for $\forall \lambda>\lambda_{b1}$,
\begin{equation*}
\begin{aligned}
&\ \ \ \ 4\|H_{11}\|_{\infty} \left\|H_{22}-H_{33}H_{13}^2-\lambda h_{22}\right\|_{\infty}-\left\|2H_{21}+H_{13}^2\right\|^2_{\infty}\\
&>4\|H_{11}\|_{\infty} \left\|H_{22}-H_{33}H_{13}^2-\lambda_{b1} h_{22}\right\|_{\infty}-\left\|2H_{21}+H_{13}^2\right\|^2_{\infty}=0.
\end{aligned}
\end{equation*}
By \cite[(2.8)]{JW},
L'Hospital Formula, and Assumption \ref{find-all-eigen-examp-ass-suf-intpr-ass},
\begin{equation*}
\begin{aligned}
  T-t_{\lambda_{b1},T}^{k_1}&= \lim_{\lambda\searrow\lambda_{b1}}
  \frac{\frac{\pi}{2}+\arctan{\frac{-\left\|2H_{21}+H_{13}^2\right\|_{\infty}}{\sqrt{4\|H_{11}\|_{\infty} \left\|H_{22}-H_{33}H_{13}^2-\lambda h_{22}\right\|_{\infty}-\left\|2H_{21}+H_{13}^2\right\|^2_{\infty}}}}}{\sqrt{\|H_{11}\|_{\infty} \left\|H_{22}-H_{33}H_{13}^2-\lambda h_{22}\right\|_{\infty}-\frac{\left\|2H_{21}+H_{13}^2\right\|^2_{\infty}}{4}}}  \\
  &=\frac{2}{\left\|2H_{21}+H_{13}^2\right\|_{\infty}}\\
  &>T,
  \end{aligned}
\end{equation*}
which implies that $t_{\lambda_{b1},T}^{k_1}<0$.
Then by \eqref{find-all-eigen-examp-ass-suf-intpr-ass-lemm-first-amonglemma3}, $t_{\lambda_{b},T}^{k_1}<0$.
Then by  \eqref{find-all-eigen-examp-ass-suf-intpr-ass-lemm-first-amonglemma4}, $t_{\lambda_{b},T}^{k}<0$.
\end{proof}

By \eqref{find-all-eigen-examp-ass-suf-intpr-ass-lemm-first-amonglemma2}, $k(\cdot;\lambda_b)$, hence $\left(k(\cdot;\lambda_b),m(\cdot;\lambda_b)\right)$, exists on the whole $[0,T]$.
Besides, since
\begin{equation*}
H_{22}(t)-H_{33}(t)H_{13}^2(t)-\lambda h_{22}(t)\le H_{22}(t)-H_{33}(t)H_{13}^2(t)-\lambda_b h_{22}(t), \quad   t\in[0,T], \quad \lambda \in [0,\lambda_b],
\end{equation*}
by Lemma \ref{comparison theorem},
\begin{equation*}
  0\le k(t;\lambda)\le k(t;\lambda_b), \indent \lambda \in [0,\lambda_b].
\end{equation*}
Then for $\forall\lambda \in [0,\lambda_b]$, $\left(k(\cdot;\lambda),m(\cdot;\lambda)\right)$ exist on the whole $[0,T]$.
Then by Lemma \ref{decouple-lemma},
there is none non-trivial solution to problem \eqref{generel-function-eigen-problem} corresponding to any $\lambda \in [0,\lambda_b]$, i.e.,
there is not any other eigenvalue in $(0,\lambda_b]$ of problem \eqref{generel-function-eigen-problem} and the $\lambda_1$ in \eqref{the-n-exist-t2npl1=0} is indeed the first positive eigenvalue of
\eqref{generel-function-eigen-problem}.
In other words,
\begin{thm}
Under  Assumption \ref{find-all-eigen-examp-ass-suf-intpr-ass},
there exists $\{\lambda_m\}_{m=1}^\infty \subset\mathbb{R}$, all those eigenvalues  of
problem (\ref{generel-function-eigen-problem}), satisfying $\lambda_m\rightarrow+\infty$ as \ $m\rightarrow +\infty$.
Moreover, the eigenfunction space corresponding to each $\lambda_m$ is $1$-dimensional.
\end{thm}

\section{The order of growth for the eigenvalues of problem (\ref{generel-function-eigen-problem})}\label{section-increasing-order}
For the eigenvalues $\{\lambda_m\}_{m=1}^{+\infty}$ in Theorem \ref{general-lambda-exist}, furthermore, we have the following
\begin{thm}   \label{increase-ratio-function-general}
Under the same conditions of Theorem \ref{general-lambda-exist}, let $\{\lambda_m\}_{m=1}^{+\infty}$ be all the eigenvalues of problem \eqref{generel-function-eigen-problem} located in $(\lambda_b,+\infty)$, then
\begin{equation*}
\lambda_m= O(m^2), \indent \text{as} \  m\to+\infty.
\end{equation*}
\end{thm}

\begin{proof}
For $\varphi=H_{21}, H_{11}, H_{22}, H_{33}, h_{22}, |H_{13}|$ and $\forall t\in[0,T]$,
denote by $\check{\varphi}$ and $\hat{\varphi}$ the constants satisfying:
\begin{eqnarray}\label{definition-h22-hat-check-22}
0<\check{H}_{11}< H_{11}(t)<\hat{H}_{11},&&\ \ \check{H}_{22}<H_{22}(t)<\hat{H}_{22}<0,  \notag\\
0\le\check{H}_{13}\le|H_{13}(t)|\le \hat{H}_{13},&&\ \ \check{H}_{33}<H_{33}(t)<\hat{H}_{33}<0,  \\
\check{H}_{21}\le H_{21}(t)\le \hat{H}_{21},&&\ \ \check{h}_{22}<h_{22}(t)<\hat{h}_{22}<0.  \notag
\end{eqnarray}
Besides, we may assume that
\begin{equation}\label{definition-h22-hat-check-11}
\hat{H}_{23}\triangleq-\check{H}_{33}\hat{H}_{13},\ \ \ \ \ \ \ \check{H}_{23}\triangleq-\hat{H}_{33}\check{H}_{13}.
\end{equation}
The proof is divided into two steps.

\textbf{Step 1.} We will prove that there are $\{\check{\lambda}_m\}$, such that $\lambda_m\leq \check{\lambda}_m$ and $\check{\lambda}_m\sim m^2$ as $m\to +\infty.$

It is easy to verify that matrix
$$\begin{bmatrix}
  \check{H}_{11}&\check{H}_{12}&\check{H}_{13}\\
  \check{H}_{21}&\check{H}_{22}&\check{H}_{23}\\
  \check{H}_{31}&\check{H}_{32}&\hat{H}_{33}
  \end{bmatrix}$$
satisfies the assumption in Theorem \ref{thm-const-increa-intro}.
We consider the following eigenvalue problem of  time-independent coefficients:
\begin{equation}\label{h22-faster-longer-system}
\left\{
\begin{aligned}
&\mathrm{d}x_t=\left[\check{H}_{21}x_t+\check{H}_{22}\left(1-\frac{\lambda \hat{h}_{22}}{\check{H}_{22}}\right)y_t+\check{H}_{23}z_t\right]\mathrm{d}t   \\
&  \indent\ \ \  +\left[\check{H}_{31}x_t+\check{H}_{32}y_t+\hat{H}_{33}z_t\right]\mathrm{d}B_t, \indent  t\in[0,T],    \\
&-\mathrm{d}y_t=\left[\check{H}_{11}x_t+\check{H}_{12}y_t+\check{H}_{13}z_t\right]\mathrm{d}t-z_t \mathrm{d}B_t,  \indent  t\in[0,T],   \\
&x(0)=0,\indent y(T) =0.
\end{aligned}
\right.
\end{equation}
Denote by $\check{\lambda}_m$ the eigenvalue of problem \eqref{h22-faster-longer-system}.
Corresponding to \eqref{h22-faster-longer-system}, similar to \eqref{general-riccati-function-case} and \eqref{general-dual-riccati-function}, for any $\bar{t}\in[0,T]$, there is a Riccati equation \eqref{h22-faster-longer-system-riccati} and a dual Riccati equation \eqref{h22-faster-longer-system-dual-riccati}:

\begin{equation}\label{h22-faster-longer-system-riccati}
  \left\{
  \begin{aligned}
&  -\frac {\mathrm{d}\check{k}}{\mathrm{d}t}=\left(2\check{H}_{21}+\check{H}_{13}^2\right)\check{k}
+\check{H}_{11}+\left(\check{H}_{22}-\hat{H}_{33}\check{H}_{13}^2-\lambda \hat{h}_{22}\right)\check{k}^2, \indent  t\le\bar{t},\\
&  \check{k}(\bar{t})=0,
  \end{aligned}
  \right.
\end{equation}
and
\begin{equation} \label{h22-faster-longer-system-dual-riccati}
\left\{
\begin{aligned}
&-\frac{\mathrm{d}\check{\tilde{k}}}{\mathrm{d}t}
=-\left(2\check{H}_{21}+\check{H}_{13}^2\right)\check{\tilde{k}}-\check{H}_{11}\check{\tilde{k}}^2
-\left(\check{H}_{22}-\hat{H}_{33}\check{H}_{13}^2-\lambda \hat{h}_{22}\right),\indent  t\le\bar{t},      \\
&\check{\tilde{k}}(\bar{t})=0.
\end{aligned}
\right.
\end{equation}
Subtracting \eqref{h22-faster-longer-system-riccati} from \eqref{general-riccati-function-case}:
\begin{equation*}
\left\{
\begin{aligned}
&-\frac {\mathrm{d}\left(k-\check{k}\right)}{\mathrm{d}t}
=\left(\check{H}_{22}-\hat{H}_{33}\check{H}_{13}^2-\lambda \hat{h}_{22}\right)
\left(k+\check{k}\right)\left(k-\check{k}\right)
+\left(2\check{H}_{21}+\check{H}_{13}^2\right) \left(k-\check{k}\right)   \\
& \indent\indent\indent\indent\ +\left[ \left(H_{22}(t)-H_{33}(t)H_{13}^2(t)-\lambda h_{22}(t)\right)-\left(\check{H}_{22}-\hat{H}_{33}\check{H}_{13}^2-\lambda \hat{h}_{22}\right) \right]k^2  \\
& \indent\indent\indent\indent \ +\left[\left(2H_{21}(t)+H_{13}^2(t)\right)-\left(2\check{H}_{21}+\check{H}_{13}^2\right)\right]k
 +H_{11}(t)-\check{H}_{11},       \indent  t\le\bar{t},  \\
&\left(k-\check{k}\right)\left(\bar{t}\right)=0.
\end{aligned}
\right.
\end{equation*}
By \eqref{moncond-H11H22H33},  $H_{11}(t)\ge\beta>0$, then by Lemma \ref{const-term-posit-equation-posit}, $k(t;\lambda)>0, t<\bar{t}$.
Moreover, by \eqref{definition-h22-hat-check-22} and \eqref{definition-h22-hat-check-11}, for $t\le\bar{t}$,
\begin{equation*}
  \begin{aligned}
&  \left[ \left(H_{22}(t)-H_{33}(t)H_{13}^2(t)-\lambda h_{22}(t)\right) -\left(\check{H}_{22}-\hat{H}_{33}\check{H}_{13}^2-\lambda \hat{h}_{22}\right) \right]k^2\left(t;\lambda\right) \ge 0,  \\
&\left[\left(2H_{21}(t)+H_{13}^2(t)\right)-\left(2\check{H}_{21}+\check{H}_{13}^2\right)\right] k\left(t;\lambda\right) \ge 0,  \\
&  H_{11}(t)-\check{H}_{11} \ge 0.
  \end{aligned}
\end{equation*}
By Lemma \ref{elementary-compare},
$k(t;\lambda)\ge\check{k}(t;\lambda),\    t\le \bar{t}$.
Besides, by $\check{H}_{11}>0$ in \eqref{definition-h22-hat-check-22} and Lemma \ref{const-term-posit-equation-posit},  $\check{k}(t;\lambda)\ge0,\ t\le \bar{t}$.
Then
\begin{equation} \label{t-check-k-t-k}
t_{\lambda,\bar{t}}^k\ge t_{\lambda,\bar{t}}^{\check{k}}.
\end{equation}
Subtracting \eqref{h22-faster-longer-system-dual-riccati} from \eqref{general-dual-riccati-function}:
\begin{equation*} 
\left\{
\begin{aligned}
&-\frac{\mathrm{d}\left(\tilde{k}-\check{\tilde{k}}\right)}{\mathrm{d}t}
= -\left(2\check{H}_{21}+\check{H}_{13}^2\right) \left(\tilde{k}-\check{\tilde{k}}\right)
    -\check{H}_{11}\left(\tilde{k}+\check{\tilde{k}}\right)\left(\tilde{k}-\check{\tilde{k}}\right) \\
&\indent\indent\indent\indent\ \  -\left[ \left(H_{22}(t)-H_{33}(t)H_{13}^2(t)\right)
    - \left(\check{H}_{22}-\hat{H}_{33}\check{H}_{13}^2\right)   \right]  \\
&\indent\indent\indent\indent\ \  -\left(H_{11}(t)- \check{H}_{11} \right) \tilde{k}^2  +\lambda\left(h_{22}(t)- \hat{h}_{22}\right) \\
&\indent\indent\indent\indent\ \  -\left[\left(2H_{21}(t)+H_{13}^2(t)\right)- \left(2\check{H}_{21}+\check{H}_{13}^2\right)\right] \tilde{k},  \indent  t\le\bar{t}, \\
&\left(\tilde{k}-\check{\tilde{k}}\right)(\bar{t})=0.
\end{aligned}
\right.
\end{equation*}
By \eqref{definition-h22-hat-check-22} and \eqref{definition-h22-hat-check-11}, for $t\le\bar{t}$,
we have
\begin{eqnarray*}
&&    -\left[ \left(H_{22}(t)-H_{33}(t)H_{13}^2(t)\right)
      - \left(\check{H}_{22}-\hat{H}_{33}\check{H}_{13}^2\right)   \right]\le0.
\end{eqnarray*}
By \eqref{lambda-b-h22-uniform-posit-quadratic-posit} and Lemma \ref{const-term-posit-equation-posit}, $$\tilde{k}(t;\lambda)\le0,\indent \check{\tilde{k}}(t; \lambda)\le0, \indent  t\le\bar{t}.$$
Moreover, for
$t\le\bar{t}$,
{\small\begin{eqnarray*}
&& -\left(H_{11}(t)- \check{H}_{11} \right) \tilde{k}^2\left(t;\lambda\right)  +\lambda\left(h_{22}(t)- \hat{h}_{22}\right)
-\left[\left(2H_{21}(t)+H_{13}^2(t)\right)- \left(2\check{H}_{21}+\check{H}_{13}^2\right)\right] \tilde{k}\left(t;\lambda\right)  \\
&&\ \ \ \ \  = -\left(\sqrt{H_{11}(t)- \check{H}_{11}}\tilde{k}\left(t;\lambda\right) +\frac{\left(2H_{21}(t)+H_{13}^2(t)\right)- \left(2\check{H}_{21}+\check{H}_{13}^2\right)}{2\sqrt{H_{11}(t)- \check{H}_{11}}}\right)^2  \\
&&\ \ \ \ \ \ \ \ \  +\left( \frac{\left(2H_{21}(t)+H_{13}^2(t)\right)- \left(2\check{H}_{21}+\check{H}_{13}^2\right)}{2\sqrt{H_{11}(t)- \check{H}_{11}}}\right)^2
  +\lambda\left(h_{22}(t)- \hat{h}_{22}\right).
\end{eqnarray*}}
Since $\left( \frac{\left(2H_{21}(t)+H_{13}^2(t)\right)- \left(2\check{H}_{21}+\check{H}_{13}^2\right)}{2\sqrt{H_{11}(t)- \check{H}_{11}}}\right)^2$
is bounded and
$h_{22}(t)- \hat{h}_{22}<0, \ t\in[0,T]$,
for sufficiently large $\lambda$,
\begin{equation*}
 -\left(H_{11}(t)- \check{H}_{11} \right) \tilde{k}^2\left(t;\lambda\right)  +\lambda\left(h_{22}(t)- \hat{h}_{22}\right)  -\left[\left(2H_{21}(t)+H_{13}^2(t)\right)- \left(2\check{H}_{21}+\check{H}_{13}^2\right)\right] \tilde{k}\left(t;\lambda\right)\le0.
\end{equation*}
By Lemma \ref{elementary-compare},
$\tilde{k}(t; \lambda) \le \check{\tilde{k}}(t; \lambda)\le0, t\le\bar{t}$.
Then for sufficiently large $\lambda$,
\begin{equation} \label{tilde-k-check-tilde-k}
t_{\lambda,\bar{t}}^{\tilde{k}}\ge t_{\lambda,\bar{t}}^{\check{\tilde{k}}}.
\end{equation}
Moreover, for $t_{\lambda,\bar{t}}^k$, $t_{\lambda,\bar{t}}^{\check{k}}$,
$t_{\lambda,\bar{t}}^{\tilde{k}}$ and $t_{\lambda,\bar{t}}^{\check{\tilde{k}}}$,
Lemma \ref{lemma-riccati-func-mon-terminal-h22} and Lemma \ref{lemma-riccati-dual-func-mon-terminal-h22} hold.
Then by \eqref{t-check-k-t-k} and \eqref{tilde-k-check-tilde-k}, thanks to the procedure in Theorem \ref{general-lambda-exist},
for large enough $\lambda$,
\begin{equation}\label{t-lambda-check-comp-pro}
  t_{2m+1+2n}(\lambda) \ge  \check{t}_{2m+1+2n}(\lambda), \quad   m=0,1,2,\cdots
\end{equation}
where $\{t_{2m-1+2n}(\cdot)\}$, $\{\check{t}_{2m-1+2n}(\cdot)\}$ are functions in \eqref{those-function-tlambda} corresponding to problem  \eqref{generel-function-eigen-problem} and \eqref{h22-faster-longer-system}   respectively.
Then for large enough index $m$,
\begin{equation}\label{h22-lambda-faster}
\lambda_m\le \check{\lambda}_m.
\end{equation}
In addition, by Theorem \ref{thm-const-increa-intro}, in problem \eqref{h22-faster-longer-system},
\begin{equation}\label{h22-lambda-faster-ass}
\frac{\check{\lambda}_m \hat{h}_{22}}{\check{H}_{22}}= O(m^2)\indent \text{as}\ \  m\rightarrow+\infty.
\end{equation}

\textbf{Step 2.} We will prove that there are $\{\hat{\lambda}_m\}$, such that $\lambda_m\geq \hat{\lambda}_m$ and $\hat{\lambda}_m\sim m^2$ as $m\to +\infty.$

It is easy to verify that for sufficiently large $-\underline{H}_{22}>0$,
matrix
$$\begin{bmatrix}
 \hat{H}_{11}&\hat{H}_{12}&\hat{H}_{13}\\
  \hat{H}_{21}&\underline{H}_{22}&\hat{H}_{23}\\
  \hat{H}_{31}&\hat{H}_{32}&\check{H}_{33}
\end{bmatrix}$$
satisfies the assumption in Theorem \ref{thm-const-increa-intro}.
Next,  consider eigenvalue problem \eqref{h22-slower-shorter-system}:
\begin{equation}    \label{h22-slower-shorter-system}
\left\{
\begin{aligned}
&\mathrm{d}x_t=\left[\hat{H}_{21}x_t+(1-\lambda)\underline{H}_{22}y_t+\hat{H}_{23}z_t\right]\mathrm{d}t +\left[\hat{H}_{31}x_t+\hat{H}_{32}y_t+\check{H}_{33}z_t\right]\mathrm{d}B_t, \ \   t\in[0,T],\\
&-\mathrm{d}y_t=\left[\hat{H}_{11}x_t+\hat{H}_{12}y_t
  +\hat{H}_{13}z_t\right]\mathrm{d}t-z_t \mathrm{d}B_t,  \indent  t\in[0,T],              \\
&x(0)=0,\indent  y(T) =0.
\end{aligned}
\right.
\end{equation}
Denote by $\underline{\lambda}_m$ the eigenvalue of problem  \eqref{h22-slower-shorter-system}.
To continue, we use the change of variable
$$\frac{\hat{H}_{22}}{\underline{H}_{22}} \left(1-\mu \frac{\check{h}_{22}}{\hat{H}_{22}}\right)=(1-\lambda).$$
Then problem \eqref{h22-slower-shorter-system} becomes \eqref{h22-slower-shorter-system-unlh-hat}:
\begin{equation}    \label{h22-slower-shorter-system-unlh-hat}
\left\{
\begin{aligned}
&\mathrm{d}x_t=\left[\hat{H}_{21}x_t+\left(\hat{H}_{22}-\mu \check{h}_{22}\right)y_t+\hat{H}_{23}z_t\right]\mathrm{d}t  \\
&\indent\indent +\left[\hat{H}_{31}x_t+\hat{H}_{32}y_t+\check{H}_{33}z_t\right]\mathrm{d}B_t,  \indent  t\in[0,T], \\
&-\mathrm{d}y_t=\left[\hat{H}_{11}x_t+\hat{H}_{12}y_t
  +\hat{H}_{13}z_t\right]\mathrm{d}t-z_t \mathrm{d}B_t,   \indent  t\in[0,T],             \\
&x(0)=0,\indent  y(T) =0.
\end{aligned}
\right.
\end{equation}
It follows that all the $\hat{\lambda}_m$ satisfying
$$\frac{\hat{H}_{22}}{\underline{H}_{22}} \left(1-\hat{\lambda}_m \frac{\check{h}_{22}}{\hat{H}_{22}}\right)=(1-\underline{\lambda}_m)$$
are eigenvalues of problem \eqref{h22-slower-shorter-system-unlh-hat}.

Corresponding to \eqref{h22-slower-shorter-system-unlh-hat},
for any $\bar{t}\in[0,T]$, there is a Riccati equation \eqref{general-riccati-function-case-unlh}
and a dual Riccati equation \eqref{general-dual-riccati-function-unlh}
in the form of \eqref{general-riccati-function-case} and \eqref{general-dual-riccati-function}:
\begin{equation}   \label{general-riccati-function-case-unlh}
\left\{
\begin{aligned}
&-\frac {\mathrm{d}\hat{k}}{\mathrm{d}t}=\left(2\hat{H}_{21}+\hat{H}_{13}^2\right)\hat{k}
+\hat{H}_{11}+\left(\left(\hat{H}_{22}-\lambda \check{h}_{22}\right)-\check{H}_{33}\hat{H}_{13}^2 \right)\hat{k}^2, \indent  t\le\bar{t},\\
&\hat{k}(\bar{t})=0,
\end{aligned}
\right.
\end{equation}
and
\begin{equation} \label{general-dual-riccati-function-unlh}
\left\{
\begin{aligned}
&-\frac{\mathrm{d}\hat{\tilde{k}}}{\mathrm{d}t}
=-\left(2\hat{H}_{21}+\hat{H}_{13}^2\right)\hat{\tilde{k}}-\hat{H}_{11}\hat{\tilde{k}}^2
-\left(\left(\hat{H}_{22}-\lambda \check{h}_{22}\right)-\check{H}_{33}\hat{H}_{13}^2 \right),\indent t\le\bar{t},      \\
&\hat{\tilde{k}}(\bar{t})=0.
\end{aligned}
\right.
\end{equation}
Subtracting \eqref{general-riccati-function-case} from \eqref{general-riccati-function-case-unlh}:
\begin{equation*}   
\left\{
\begin{aligned}
&-\frac {\mathrm{d}\left(\hat{k}-k\right)}{\mathrm{d}t}
= \left(H_{22}(t)-H_{33}(t)H_{13}^2(t)-\lambda h_{22}(t)\right)\left(\hat{k}+k\right)\left(\hat{k}-k\right)  \\
&\indent\indent\indent\indent\ \ +\left(2H_{21}(t)+H_{13}^2(t)\right)\left(\hat{k}-k\right)    \\
&\indent\indent\indent\indent\ \ +\left[\left(\hat{H}_{22}-\lambda \check{h}_{22}-\check{H}_{33}\hat{H}_{13}^2 \right)
     - \left(H_{22}(t)-H_{33}(t)H_{13}^2(t)-\lambda h_{22}(t)\right) \right]\hat{k}^2    \\
&\indent\indent\indent\indent\ \ +\left[\left(2\hat{H}_{21}+\hat{H}_{13}^2\right)- \left(2H_{21}(t)+ H_{13}^2(t)\right)\right]\hat{k}  \\
&\indent\indent\indent\indent\ \ +  \hat{H}_{11}-{H}_{11}(t), \indent
      t\le\bar{t},   \\
&\left(\hat{k}-k\right)(\bar{t})=0.
\end{aligned}
\right.
\end{equation*}
By \eqref{definition-h22-hat-check-22}, $\hat{H}_{11}>0$, then  by Lemma \ref{const-term-posit-equation-posit}, $\hat{k}(t;\lambda)\ge0, t\le\bar{t}$,
and similarly, $k(t;\lambda)\ge0, t\le\bar{t}$.
Then by \eqref{definition-h22-hat-check-22} and \eqref{definition-h22-hat-check-11}, for $t\le \bar{t}$,
\begin{equation*}
\begin{aligned}
&\left[\left(\hat{H}_{22}-\lambda \check{h}_{22}-\check{H}_{33}\hat{H}_{13}^2 \right)
     - \left(H_{22}(t)-H_{33}(t)H_{13}^2(t)-\lambda h_{22}(t)\right) \right]\hat{k}^2\left(t;\lambda\right)\ge0,\\
&\left[\left(2\hat{H}_{21}+\hat{H}_{13}^2\right)- \left(2H_{21}(t)+H_{13}^2(t)\right)\right]\hat{k}\left(t;\lambda\right)\ge0,\\
&\hat{H}_{11}-{H}_{11}(t) \ge0.
\end{aligned}
\end{equation*}
By Lemma \ref{elementary-compare},
$\hat{k}(t;\lambda)\ge k(t;\lambda)\ge0,\ t\le\bar{t}$,
whence
\begin{equation}\label{t-hat-k-k-h22}
t_{\lambda,\bar{t}}^{\hat{k}}\ge t_{\lambda,\bar{t}}^k.
\end{equation}

Subtracting \eqref{general-dual-riccati-function} from \eqref{general-dual-riccati-function-unlh}:
\begin{equation*}  
\left\{
\begin{aligned}
&-\frac{\mathrm{d}\left(\hat{\tilde{k}}-\tilde{k}\right)}{\mathrm{d}t}
=  -\left(2\hat{H}_{21}+\hat{H}_{13}^2\right) \left(\hat{\tilde{k}}-\tilde{k}\right)
    -\hat{H}_{11} \left(\hat{\tilde{k}}+\tilde{k}\right)\left(\hat{\tilde{k}}-\tilde{k}\right)     \\
&\indent\indent\indent\indent\ \  -\left[ \left(\hat{H}_{22}-\check{H}_{33}\hat{H}_{13}^2\right) - \left(H_{22}(t)-H_{33}(t)H_{13}^2(t)\right) \right]    \\
&\indent\indent\indent\indent\ \  - \left[\hat{H}_{11}-H_{11}(t) \right] \tilde{k}^2+\lambda \left( \check{h}_{22}- h_{22}(t) \right)   \\
&\indent\indent\indent\indent\ \  -\left[ \left(2\hat{H}_{21}+\hat{H}_{13}^2\right) -\left(2H_{21}(t)+H_{13}^2(t)\right) \right] \tilde{k},
     \indent  t\le\bar{t},   \\
&\left(\hat{\tilde{k}}-\tilde{k}\right)(\bar{t})=0.
\end{aligned}
\right.
\end{equation*}
By \eqref{definition-h22-hat-check-22} and \eqref{definition-h22-hat-check-11}, for $t\le\bar{t}$:
\begin{eqnarray*}
&&-\left[\left(\hat{H}_{22}-\check{H}_{33}\hat{H}_{13}^2\right)
-\left(H_{22}(t)-H_{33}(t)H_{13}^2(t)\right)\right]\le0.
\end{eqnarray*}
Moreover, for $t\le\bar{t}$,
{\small\begin{eqnarray*}
&&  - \left[\hat{H}_{11}-H_{11}(t) \right] \tilde{k}^2\left(t;\lambda\right)+\lambda \left( \check{h}_{22}- h_{22}(t) \right) -\left[ \left(2\hat{H}_{21}+\hat{H}_{13}^2\right) -\left(2H_{21}(t)+H_{13}^2(t)\right) \right] \tilde{k}\left(t;\lambda\right) \\
&&\ \ \ \ \ \  = -\left( \sqrt{\hat{H}_{11}-H_{11}(t)} \tilde{k}\left(t;\lambda\right)
    +\frac{\left(2\hat{H}_{21}+\hat{H}_{13}^2\right) -\left(2H_{21}(t)+H_{13}^2(t)\right)}{2\sqrt{\hat{H}_{11}-H_{11}(t)}}   \right)^2  \\
&&\ \ \ \ \ \ \ \ \  +\left( \frac{\left(2\hat{H}_{21}+\hat{H}_{13}^2\right) -\left(2H_{21}(t)+H_{13}^2(t)\right) } {2\sqrt{\hat{H}_{11}-H_{11}(t)}} \right)^2
    +\lambda \left( \check{h}_{22}- h_{22}(t) \right).
\end{eqnarray*}}
Since $\left( \frac{\left(2\hat{H}_{21}+\hat{H}_{13}^2\right) -\left(2H_{21}(t)+H_{13}^2(t)\right) } {2\sqrt{\hat{H}_{11}-H_{11}(t)}} \right)^2$
is bounded and
$\check{h}_{22}- h_{22}(t)<0$,
$\forall t\in[0,T]$,
for sufficiently large $\lambda>0$,
\begin{equation*}
  - \left[\hat{H}_{11}-H_{11}(t) \right] \tilde{k}^2\left(t;\lambda\right)+\lambda \left( \check{h}_{22}- h_{22}(t) \right)  -\left[ \left(2\hat{H}_{21}+\hat{H}_{13}^2\right) -\left(2H_{21}(t)+H_{13}^2(t)\right) \right] \tilde{k}\left(t;\lambda\right) \le0.
\end{equation*}
By Lemma \ref{elementary-compare},
$\hat{\tilde{k}}\left(t;\lambda\right)\le\tilde{k}\left(t;\lambda\right),\ t\le\bar{t}$.
Besides,
by \eqref{lambda-b-h22-uniform-posit-quadratic-posit} and Lemma \ref{const-term-posit-equation-posit}, $\tilde{k}(t;\lambda)\le0, t\le\bar{t}$.
Hence when $\lambda$ is large enough,
\begin{equation} \label{t-hat-tilde-k-tilde-k}
t_{\lambda,\bar{t}}^{\hat{\tilde{k}}} \ge t_{\lambda,\bar{t}}^{\tilde{k}}.
\end{equation}

Moreover, for $t_{\lambda,\bar{t}}^k$, $t_{\lambda,\bar{t}}^{\hat{k}}$,
$t_{\lambda,\bar{t}}^{\tilde{k}}$ and $t_{\lambda,\bar{t}}^{\hat{\tilde{k}}}$,
Lemma \ref{lemma-riccati-func-mon-terminal-h22} and Lemma \ref{lemma-riccati-dual-func-mon-terminal-h22} hold.
Then by \eqref{t-hat-k-k-h22}, \eqref{t-hat-tilde-k-tilde-k}, and Theorem \ref{general-lambda-exist}, as we have done in \eqref{t-lambda-check-comp-pro}, for sufficiently large index $m$,
\begin{equation}\label{h22-lambda-slower}
\hat{\lambda}_m \le \lambda_m,
\end{equation}
where  \vskip -1cm
\begin{eqnarray}\label{rela-hatlambda-lambda}
\hat{\lambda}_m= \left( 1-\frac{\underline{H}_{22}}{\hat{H}_{22}}\left( 1-\underline{\lambda}_m \right) \right)
\frac{\hat{H}_{22}}{\check{h}_{22}}, \indent   m=1,2,3,\cdots
\end{eqnarray}
and $\underline{\lambda}_m$ is the eigenvalue of problem \eqref{h22-slower-shorter-system}.
By Theorem \ref{thm-const-increa-intro} again,
\begin{equation}\label{h22-lambda-slower-ass}
\underline{\lambda}_m= O(m^2), \indent  \text{as} \ \  m\rightarrow +\infty,
\end{equation}
and hence
$$
\hat{\lambda}_m= O(m^2),  \indent  \text{as} \ \  m\rightarrow +\infty.
$$
By \eqref{h22-lambda-faster}, \eqref{h22-lambda-faster-ass}, \eqref{h22-lambda-slower}, \eqref{h22-lambda-slower-ass},
\begin{equation*}\label{increase-ratio-function-general-h22-th}
\lambda_m= O(m^2), \indent \text{as} \ \  m\to+\infty.
\end{equation*}
\end{proof}

\begin{rem}In particular, we take $h_{22}=H_{22}$ in \eqref{generel-function-eigen-problem}.
Then eigenvalue problem \eqref{generel-function-eigen-problem} becomes the following one:
\begin{equation}\label{generel-function-eigen-problem-coro-H22}
  \left\{
  \begin{aligned}
& \mathrm{d}x_t=\left[H_{21}x_t+(1-\lambda)H_{22}y_t+H_{23}z_t\right]\mathrm{d}t +\left[H_{31}x_t+H_{32}y_t+H_{33}z_t\right]\mathrm{d}B_t,   \ \  t\in[0,T], \\
&  -\mathrm{d}y_t=\left[H_{11}x_t+H_{12}y_t+H_{13}z_t\right]\mathrm{d}t-z_t \mathrm{d}B_t, \indent  t\in[0,T],\\
& x(0)=0,\indent  y(T) =0,
  \end{aligned}
  \right.
\end{equation}
which is a generalization of \cite[(3.3)]{peng}, from constant-coefficients case to time-dependent case.
\end{rem}
\begin{ass} \label{assumption-1d-general-perturbation-cor}
Assume that $n=1$ and $H_{ij}\in C([0,T],\mathbb{R})$, $i,j=1,2,3$. Besides, $H$ satisfy (\ref{moncondition}) uniformly for $t\in[0,T]$.
Moreover,
\begin{equation*}
H_{23}(t)=-H_{33}(t)H_{13}(t),\indent  t\in[0,T].
\end{equation*}
\end{ass}
\begin{cor}\label{conclu-h22-H22-coroll}
Under Assumption \ref{assumption-1d-general-perturbation-cor},
there exists $\{\lambda_m\}_{m=1}^\infty \subset(\lambda_b,+\infty)$, all the eigenvalues of
problem \eqref{generel-function-eigen-problem-coro-H22} which are contained in $(\lambda_b,+\infty)$, such that $\lambda_m\rightarrow+\infty$ as \ $m\rightarrow +\infty$.
Moreover, the eigenfunction space corresponding to each $\lambda_m$ is of  $1$-dimensional.
Besides,
\begin{equation*}
\lambda_m= O(m^2), \indent \text{as}\  m\to+\infty.
\end{equation*}
\end{cor}
Corollary \ref{conclu-h22-H22-coroll} is a further study of \cite[Theorem 3.2]{peng}.

\section{Application: estimation of statistic period of solutions to Forward-Backward SDEs}\label{section-estimation-period}
In this section, based on Theorem \ref{general-lambda-exist} and Theorem \ref{increase-ratio-function-general}, together with Proposition \ref{con-prop-lambd-estima}, we can  estimate the statistic period of solutions of FBSDEs  directly by its time-dependent coefficients and time duration:
\begin{thm}\label{thm-estim-period-coef}
  Let $\lambda$  be an eigenvalue in Theorem \ref{general-lambda-exist}, then for sufficiently large $m\in\mathbb{N}_+$, for
\begin{equation*}
 \lambda>\frac{4\pi^2 m^2}{-\check{H}_{11}\hat{h}_{22} T^2},\indent
           \left(resp. \quad  \lambda< \frac{\hat{H}_{22}-\underline{H}_{22}}{\check{h}_{22}}
                                        + \frac{\pi^2 m^2}{-2\hat{H}_{11}\check{h}_{22} T^2} \right)
\end{equation*}
 the statistic period of associate eigenfunctions is greater (resp. less) than $m$.
\end{thm}
\begin{proof}
By \eqref{h22-lambda-faster} and \eqref{h22-lambda-slower},
\begin{equation}\label{order-hatlambd-lambd-checklambd}
  \hat{\lambda}_m\le\lambda_m\le\check{\lambda}_m.
\end{equation}
On the one hand, applying Proposition \ref{con-prop-lambd-estima} to system \eqref{h22-faster-longer-system},
\begin{equation}\label{lambda-check-up-bound}
 \check{\lambda}_m\le\frac{4\pi^2 m^2}{-\check{H}_{11}\hat{h}_{22} T^2}.
\end{equation}
On the other hand, applying Proposition \ref{con-prop-lambd-estima} to system  \eqref{h22-slower-shorter-system},
\begin{equation*}
  \underline{\lambda}_m\ge\frac{\pi^2 m^2}{-2\hat{H}_{11}\underline{H}_{22} T^2}.
\end{equation*}
Together with relation \eqref{rela-hatlambda-lambda},
\begin{equation}\label{rela-lambdahat-undll-low-bound}
  \begin{aligned}
  \hat{\lambda}_m&= \frac{\hat{H}_{22}-\underline{H}_{22}}{\check{h}_{22}}
           +\frac{\underline{H}_{22}}{\check{h}_{22}}\underline{\lambda}_m \\
& \ge \frac{\hat{H}_{22}-\underline{H}_{22}}{\check{h}_{22}} + \frac{\pi^2 m^2}{-2\hat{H}_{11}\check{h}_{22} T^2}.
  \end{aligned}
\end{equation}
From \eqref{order-hatlambd-lambd-checklambd}, \eqref{lambda-check-up-bound} and \eqref{rela-lambdahat-undll-low-bound}, we have
\begin{equation}\label{estimate-bound-lambda-m-ulti}
\frac{\hat{H}_{22}-\underline{H}_{22}}{\check{h}_{22}}  + \frac{\pi^2 m^2}{-2\hat{H}_{11}\check{h}_{22} T^2} \le \lambda_m \le \frac{4\pi^2 m^2}{-\check{H}_{11}\hat{h}_{22} T^2}.
\end{equation}
Besides, by combining the procedure in \cite[Section 6.1, Proof of Theorem 3.2]{peng} for time-invariant case and the auxiliary systems in our proof of Theorem \ref{increase-ratio-function-general}, the statistic period of eigenfunctions associated with $\lambda_m$ is $m$, from which and \eqref{estimate-bound-lambda-m-ulti} the proof is finished. Note that the index $m$ of $\lambda_m$ in this section differs from the one in the proof of Theorem \ref{general-lambda-exist} and the difference of them is $n$.

\end{proof}

\section{Example illustrating the subtle difficulty emerges in the time-dependent problem}\label{illustrating unusual example}\label{exa-new-case-dual-to0}
In time-dependent case, when $\lambda$ is not large enough or perturbation $h_{22}$ is degenerate, things can be quite subtle. We show how situation goes different by the following meticulous  example.

Consider a stochastic Hamiltonian system with the following coefficients $H$:
\begin{eqnarray}\label{hamiltonian-matrix-example}
  H=\begin {bmatrix}
  3&0&1\\
  0&-4&2\\
  1&2&-2
  \end{bmatrix}.
\end{eqnarray}
For any $\xi=(x,y,z)\in \mathbb{R}^3$, we have
\begin{equation*}
  \begin {bmatrix}
  x&y&z
 \end{bmatrix}
\begin {bmatrix}
  -3&0&-1\\
  0&-4&2\\
  1&2&-2
 \end{bmatrix}
\begin {bmatrix}
  x\\
  y\\
  z
 \end{bmatrix}
\le -3x^2-\left[4-16\left(\sqrt{2}-1\right)^2\right]y^2-\left[2-\frac{\left(\sqrt{2}+1\right)^2}{4}\right]z^2.
\end{equation*}
Taking $\beta=2-\frac{(\sqrt{2}+1)^2}{4}$(not optimal but not matter),  \eqref{hamiltonian-matrix-example} satisfies monotonicity condition (\ref{moncondition}).

The following technical lemma is needed.
\begin{lem}\label{example-dual-Riccati-lemma-techn}
The first negative root of the solution to the following equation
\begin{equation} \label{example-dual-Riccati}
\left\{
\begin{aligned}
&-\frac {\mathrm{d}\tilde{k}}{\mathrm{d}t}=-\tilde{k}-3\tilde{k}^2-30t-1,\indent   t\le0,   \\
&\tilde{k}(0)=0,
\end{aligned}
\right.
\end{equation}
uniquely exists, denoted as $-T_1$.
\end{lem}
Besides, denote
$$T_2 =\left(\frac{\pi}{2}-\arctan\frac{1}{\sqrt{11}}\right)\frac{2}{\sqrt{11}},$$
$$T=T_1+T_2.$$

\begin{proof}[Proof of Lemma \ref{example-dual-Riccati-lemma-techn}]

Note that the purpose is to find out the maximal negative root of $\tilde{k}$ (if it exists). Once it appears, the following procedure stops immediately. Otherwise, the following procedure can be carried out properly with $\tilde{k}<0$.

$-\frac {\mathrm{d}\tilde{k}}{\mathrm{d}t}\Big|_{t=0}=-1<0$.
When $\tilde{k}\in[-1,0]$ and $t\le0$,
$$\frac {\mathrm{d}\tilde{k}}{\mathrm{d}t}=\tilde{k}+3\tilde{k}^2+30t+1\le 4.$$
Then for $t\in\left[-\frac{1}{4},0\right]$,
$\tilde{k}\ge 4t \ge-1$.
In particular, $\tilde{k}\big|_{t=-\frac{1}{4}}\ge-1.$
However, for $\forall t\le-\frac{1}{4}$  and $\tilde{k}\in[-1,0]$,
$$\frac {\mathrm{d}\tilde{k}}{\mathrm{d}t}=\tilde{k}+3\tilde{k}^2+30t+1 \le -\frac{7}{2}<0,$$
and
$$\tilde{k}(t)\Big|_{t\le-\frac{1}{4}}\ge \left(-\frac{7}{2}t-\frac{15}{8}\right)\bigg|_{t\le-\frac{1}{4}}\ge-1.$$
Besides, $$\left(-\frac{7}{2}t-\frac{15}{8}\right)\bigg|_{t=-\frac{15}{28}<-\frac{1}{4}}=0.$$
So the maximal negative root of $\tilde{k}$  exists indeed.
\end{proof}

Consider the following eigenvalue problem of stochastic Hamiltonian system with boundary conditions:
\begin{equation}  \label{example-eigen-problem}
\left\{
\begin{aligned}
&\mathrm{d}x_t=[(-4-\lambda h_{22})y_t+2z_t]\mathrm{d}t+[x_t+2y_t-2z_t]\mathrm{d}B_t, \indent   t\in[0,T],   \\
&-\mathrm{d}y_t=[3x_t+z_t]\mathrm{d}t-z_t \mathrm{d}B_t,\indent   t\in[0,T],    \\
&x(0)=0,\indent   y(T) =0,
\end{aligned}
\right.
\end{equation}
where
\begin{equation*} {h_{22}(t)=}
\left\{
\begin{aligned}
     -10(t-T_{1})-1,&\indent  t\in[0,T_{1}),  \\
        -1,&\indent t\in[T_{1}, T].
\end{aligned}
\right.
\end{equation*}


Corresponding to problem \eqref{example-eigen-problem},  Riccati equation (\ref{general-riccati-function-case}) becomes
\begin{equation}\label{example-Riccati}
\left\{
\begin{aligned}
&-\frac {\mathrm{d}k}{\mathrm{d}t}=k+3+(-2-\lambda h_{22})k^2, \indent   t\le T,         \\
&k(T)=0.
\end{aligned}
\right.
\end{equation}
The solution to (\ref{example-Riccati}) with $\lambda=3$ is
\begin{equation*}
  k=\frac{\sqrt{11}}{2}\tan\left[\frac{\sqrt{11}}{2}\left(T-t\right)+\arctan\frac{1}{\sqrt{11}}\right]-\frac{1}{2}.
\end{equation*}
The length of existing interval of $k$ is $T_2$.
Then $(k,m)$ exist on $[T_1,T]$.


On the other hand,  dual Riccati equation (\ref{general-dual-riccati-function}) becomes
\begin{equation}\label{general-dual-riccati-function-new-exam-dual}
\left\{
\begin{aligned}
 &-\frac {\mathrm{d}\tilde{k}}{\mathrm{d}t}=-\tilde{k}-3\tilde{k}^2+\left(2+\lambda h_{22}(t)\right),\indent  t\le T_1, \\
 &\tilde{k}(T_1)=0.
 \end{aligned}
 \right.
\end{equation}
After taking $\lambda=3$,
\eqref{general-dual-riccati-function-new-exam-dual} becomes:
\begin{equation}\label{general-dual-riccati-function-new-exam-dual-simp}
\left\{
\begin{aligned}
& -\frac {\mathrm{d}\tilde{k}}{\mathrm{d}t}=-\tilde{k}-3\tilde{k}^2-30(t-T_1)-1,\indent   t\le T_1, \\
& \tilde{k}(T_1)=0.
 \end{aligned}
 \right.
\end{equation}
Since the property of solution to \eqref{general-dual-riccati-function-new-exam-dual-simp} coincides with that of \eqref{example-dual-Riccati},  by Lemma \ref{example-dual-Riccati-lemma-techn},
$\tilde{k}$, thus $(\tilde{k},\tilde{m})$, exist on $[0,T_1]$ and $\tilde{k}(0)=0$.

Furthermore, by Lemma \ref{decouple-lemma}, with $\lambda=3$ and any given $y(0)\not=0$, \eqref{example-eigen-problem} has unique nontrivial solution.
Then $\lambda=3$ is an eigenvalue of \eqref{example-eigen-problem}.

Corresponding eigenfunctions can be constructed as follows.
Firstly, solve stochastic differential equation with initial condition:
\begin{equation*}
\left\{
\begin{aligned}
&\mathrm{d}\widetilde{x}_t=\left[-1-\frac{7}{2}\tilde{k}+\frac{1}{2}\tilde{m}\right]\widetilde{x}_t\mathrm{d}t
+ \left[1+\frac{1}{2}\tilde{k}-\frac{1}{2}\tilde{m} \right]\widetilde{x}_t\mathrm{d}B_t, \indent   t\in\left[0,\frac{T_1+T_2}{2}\right],  \\
&\widetilde{x}(0)= \tilde{x}_0\not=0,
\end{aligned}
\right.
\end{equation*}
where $(\tilde{k},\tilde{m})$ is continuous and bounded on $\left[0,\frac{T_1+T_2}{2}\right]$,
then we get $(\tilde{x},\tilde{y},\tilde{z})$ on $\left[0,\frac{T_1+T_2}{2}\right]$ and $x(0)=\tilde{y}(0)=\tilde{k}(0)\tilde{x}(0)=0$.
Let $x(\frac{T_1+T_2}{2})=\tilde{y}(\frac{T_1+T_2}{2})$.
Then solve the following equation
\begin{equation*}
\left\{
\begin{aligned}
&\mathrm{d}x_t=[-\left(4+3h_{22}(t)\right)k+ 2m]x_t\mathrm{d}t + \left[1+2k-2m\right]x_t\mathrm{d}B_t,\indent
t\in\left[\frac{T_1+T_2}{2},T\right],    \\
&x\left(\frac{T_1+T_2}{2} \right)=\widetilde{y}\left(\frac{T_1+T_2}{2}\right),
\end{aligned}
\right.
\end{equation*}
where $(k,m)$ is continuous and bounded on $\left[\frac{T_1+T_2}{2},T\right]$.
Then we get $(x,y,z)$ on $\left[\frac{T_1+T_2}{2},T\right]$ and $y(T)=k(T)x(T)=0$.
Moreover, eigenfunction $(x,y,z),\ t\in[0,T]$ comes from the similar method in \eqref{xyz-form-dual}.

\begin{rem}
In the case depicted by this example, eigenvalue emerges when corresponding solution to dual Riccati equation get back to zero, rather than    blowing  up to $+\infty$.
This case is beyond the scope of \cite{peng} and this paper.
To characterize the property of $t_\lambda^k$ and $t_\lambda^{\tilde{k}}$ in $\lambda$ is far from easy.
Further   study demands new methods.
\end{rem}


\section{Appendix}

\subsection{Example denying the naive expectation}\label{examp-deny-naive-expectation}
\begin{exam}\label{exam-H11}
In \eqref{evp2}, assume that $n=1$ and
$$
\bar{H}=\begin {bmatrix}
  -H_{11}&0&0\\
  0&0&0\\
  0&0&0
  \end{bmatrix},
$$
then the eigenvalue problem is rewritten as

\begin{equation} \label{multi-ei-v-problem-II}
\left\{
\begin{aligned}
&\mathrm{d}x_t=[H_{21}x_t+ H_{22}y_t+ H_{23}z_t]\mathrm{d}t  +[H_{31}x_t+ H_{32}y_t+ H_{33}z_t]\mathrm{d}B_t, \indent   t\in[0,T],  \\
&-\mathrm{d}y_t=[(1+\l)H_{11}x_t+H_{12}y_t+ H_{13}z_t]\mathrm{d}t-z_t \mathrm{d}B_t, \indent   t\in[0,T], \\
&x(0)=0,\indent   y(T) =0.
\end{aligned}
\right.
\end{equation}

On  one hand, since $\bar{H}$ is negative, by \cite[Section 7]{peng}, all the eigenvalues of \eqref{multi-ei-v-problem-II} should be positive. On the other hand, apparently, for $\lambda\ge0$, the coefficient matrix of \eqref{multi-ei-v-problem-II} satisfies the monotonicity condition \eqref{moncondition}, i.e.,
there is a constant  $\beta>0$, such that for any $(x,y,z)\in\mathbb{R}^{3}$,
$$\begin {bmatrix}
  x&y&z
  \end{bmatrix}
\begin {bmatrix}
  -(1+\lambda)H_{11}&-H_{12}&-H_{13}\\
  H_{21}&H_{22}&H_{23}\\
  H_{31}&H_{32}&H_{33}
  \end{bmatrix}
  \begin {bmatrix}
  x\\y\\z
  \end{bmatrix}
  \le-\beta \left(x^2+y^2+z^2\right).
  $$
By \cite[Proposition 2.1]{peng},  the solution $(x,y,z)$ of \eqref{multi-ei-v-problem-II} is unique,  which is $(0,0,0)$.
Then any $\lambda\in \mathbb{R}$ is not an eigenvalue of problem \eqref{multi-ei-v-problem-II}.

Now, if feasible, by taking  expectation on the stochastic Hamiltonian system  \eqref{multi-ei-v-problem-II}, we have the following ODE with two-points boundary conditions:
\begin{equation}\label{ODE-take-expect-eige-exam}
\left\{
\begin{aligned}
&\mathrm{d}{\mathbb E}x_t=[H_{21}\mathbb Ex_t+ H_{22}\mathbb Ey_t+ H_{23}\mathbb Ez_t]\mathrm{d}t, \indent   t\in[0,T],  \\
&-\mathrm{d}\mathbb Ey_t=[(1+\l)H_{11}\mathbb Ex_t+H_{12}\mathbb Ey_t+ H_{13}\mathbb Ez_t]\mathrm{d}t, \indent   t\in[0,T], \\
&\mathbb Ex(0)=0,\indent  \mathbb E y(T) =0.
\end{aligned}
\right.
\end{equation}
Assume $H_{13}(t)\not=0,\ \forall t\in[0,T]$ and let $\mathbb Ey_t\equiv0$, $\mathbb Ez_t=-H_{13}^{-1}(t)(1+\l)H_{11}\mathbb Ex_t$.
Then obviously all the $\lambda\in\mathbb R$ are  eigenvalues of \eqref{ODE-take-expect-eige-exam}. This means that, we can not study the eigenvalue problem of stochastic Hamiltonian system merely by taking expectation.
\end{exam}

\subsection{Proof of Lemma \ref{decouple-lemma}}
\begin{proof}[Proof of Lemma \ref{decouple-lemma}]
By differentiating $K(\cdot)x(\cdot)$ directly, we get that
$$(x(t), K(t)x(t), M(t)x(t)), \indent   t\in[T_1,T_2]$$
is exactly a solution to  (\ref{stochastic-Hamilton}).
We now prove the uniqueness.

Let $(x(t),y(t),z(t)),\ t\in[T_1,T_2]$ be a solution to equation (\ref{stochastic-Hamilton}), and denote $$(\bar{y}_t,\bar{z}_t)=(K_tx_t,M_tx_t), \indent (\hat{y},\hat{z})=(\bar{y}_t-y_t,\bar{z}_t-z_t).$$
Differentiating $\bar{y}_t$:
\begin{equation*}
\begin{aligned}
-\mathrm{d}\bar{y}_t&=-\emph{\text{\.{K}}}_tx_t\mathrm{d}t-K_t\mathrm{d}x_t    \\
&=[K_t(H_{22}\bar{y}_t+H_{23}\bar{z}_t)+H_{11}x_t+H_{12}\bar{y}_t+H_{13}\bar{z}_t]\mathrm{d}t  \\
&\ \ \    -K_t(H_{22}y_t+H_{23}z_t)\mathrm{d}t-K_t(H_{31}x_t+H_{32}y_t+H_{33}z_t)\mathrm{d}B_t  \\
&=[K_t(H_{22}\hat{y}_t+H_{23}\hat{z}_t)+H_{11}x_t+H_{12}\bar{y}_t+H_{13}\bar{z}_t]\mathrm{d}t   \\
&\ \ \   -K_t(H_{31}x_t+H_{32}y_t+H_{33}z_t)\mathrm{d}B_t.
\end{aligned}
\end{equation*}
To equation $M=K( H_{31}+H_{32}K+H_{33}M)$, multiplied by $x_t$ from right, we have
$K_tH_{31}x_t=\bar{z}_t-K(H_{32}\bar{y}_t+H_{33}\bar{z}_t)$.
Then
\begin{equation*}
\begin{aligned}
-\mathrm{d}\hat{y}_t& =[K(H_{22}\hat{y}_t+H_{23}\hat{z}_t)+H_{12}\hat{y}_t+H_{13}\hat{z}_t]\mathrm{d}t \\
&\ \ \   -[\hat{z}_t-K_t(H_{32}\hat{y}_t+H_{33}\hat{z}_t)]\mathrm{d}B_t          \\
&=[(H_{12}+KH_{22})\hat{y}_t+(H_{13}+KH_{23})\hat{z}_t]\mathrm{d}t  \\
&\ \ \  -[-K_tH_{32}\hat{y}_t+(I_n-KH_{33})\hat{z}_t]\mathrm{d}B_t, \indent   t\le T_2,  \\
\hat{y}(T_2)&=0.
\end{aligned}
\end{equation*}
Let $z_t'=-K_tH_{32}\hat{y}_t+(I_n-KH_{33})\hat{z}_t$.
Assume that $I_n-KH_{33}$ is invertible and its inverse is uniformly bounded. Then $\hat{z}_t=(I_n-KH_{33})^{-1}(z_t'+K_tH_{32}\hat{y}_t)$, and the above equation becomes
\begin{equation*}
\left\{
\begin{aligned}
&-\mathrm{d}\hat{y}_t =\left[(H_{12}+KH_{22})\hat{y}_t+(H_{13}+KH_{23})(I_n-KH_{33})^{-1}(z_t'+K_tH_{32}\hat{y}_t)\right]\mathrm{d}t \\
&\indent\indent\ \ -z_t'\mathrm{d}B_t,  \indent   t\le T_2,      \\
&\hat{y}(T_2)=0.
\end{aligned}
\right.
\end{equation*}
The above typical  linear backward stochastic differential equation has a unique solution $(\hat{y}_t,z'_t)\equiv0$.
Thus, $y_t=\bar{y}_t=K_tx_t,\ z_t=\bar{z}_t=M_tx_t$.

If $I_n-KH_{33}$ is not invertible, by using It\^{o}'s formula to $|\hat{y}|^2$ directly, we have
\begin{equation*}
\begin{aligned}
  & \mathbb{E}|\hat{y}_t|^2+\mathbb{E}\int_t^{T_2} \langle(I_n-KH_{33})\hat{z}_s, (I_n-KH_{33})\hat{z}_s\rangle  \mathrm{d}s  \notag\\
  &\ \ \ \ \ \ = 2\mathbb{E}\int_t^{T_2} \langle \hat{y}_s, (H_{12}+KH_{22})\hat{y}_s  \rangle\mathrm{d}s
     -2\mathbb{E}\int_t^{T_2} \langle KH_{32}\hat{y}_s, KH_{32}\hat{y}_s\rangle \mathrm{d}s \notag  \\
  &\ \ \ \ \ \ \ \ \ +2\mathbb{E} \int_t^{T_2} \langle KH_{32}\hat{y}_s, (I_n-KH_{33})\hat{z}_s \rangle  \mathrm{d}s
     +2\mathbb{E} \int_t^{T_2} \langle \hat{y}_s, (H_{13}+KH_{23})\hat{z}_s \rangle  \mathrm{d}s\\
  &\ \ \ \ \ \ \le \left(1+\frac{1}{c_2}+\frac{1}{c_1}\max _{t\in[T_1,T_2]}\|KH_{32}\|^2+\max _{t\in[T_1,T_2]}\|H_{12}+KH_{22}\|^2\right)
    \mathbb{E}\int_t^{T_2} \langle \hat{y}_s, \hat{y}_s \rangle  \mathrm{d}s   \notag  \\
  &\ \ \ \ \ \ \ \ \  +c_1\mathbb{E}\int_t^{T_2} \langle (I_n-KH_{33})\hat{z}_s, (I_n-KH_{33})\hat{z}_s \rangle \mathrm{d}s   \notag\\
  &\ \ \ \ \ \ \ \ \ +c_2\mathbb{E}\int_t^{T_2} \left\langle (H_{13}+KH_{23})\hat{z}_s, (H_{13}+KH_{23})\hat{z}_s \right\rangle \mathrm{d}s,   \notag
        \indent   t\in[T_1,T_2],
\end{aligned}
\end{equation*}
where $c_1, c_2>0$ can be adjusted.
In addition, both $K$ and $H_{ij}, i,j=1,2,3$ are continuous on $[T_1,T_2]$, so their maximum of norm exists.
As a result, if the condition (\ref{weak-unique-condition}) is satisfied with some $c>0$, we can
take $c_1=\frac{1}{2}$ and $c_2=\frac{c}{4}$, resulting in $\hat{y}=\hat{z}=0$ from Gronwall inequality.
This proves the uniqueness.
\end{proof}

\subsection{Proof of Lemma \ref{lemma-k-tilde-function-h22}}
\begin{proof}[proof of Lemma \ref{lemma-k-tilde-function-h22}]
Firstly, we prove \eqref{limit-tilde-t-lambda-k-h22-1}.
Denote by $\tilde{k}_1(\cdot;\lambda)$ the solution to the following equation:
\begin{equation} \label{general-dual-riccati-function-k1}
\left\{
\begin{aligned}
&-\frac{\mathrm{d}\tilde{k}_1}{\mathrm{d}t}
=-\left(2H_{21}+H_{13}^2\right)\tilde{k}_1-\beta\tilde{k}_1^2
-\left(H_{22}-H_{33}H_{13}^2-\lambda \hat{h}_{22}\right),\indent   t\le\bar{t},      \\
&\tilde{k}_1(\bar{t})=0.
\end{aligned}
\right.
\end{equation}
By \eqref{moncond-H11H22H33} and Assumption \ref{assumption-1d-general-perturbation}, for $t\in[0,T]$,
$$-H_{11}(t)\le-\beta<0\ \ \text{ and }\ \   h_{22}(t)\le \hat{h}_{22}<0.$$
Applying Lemma \ref{comparison theorem},
$$\tilde{k}(t;\lambda)\le\tilde{k}_1(t;\lambda), \indent   t\le\bar{t}.$$
Denote by $\tilde{k}_2(\cdot;\lambda)$ the solution to the following equation:
\begin{equation} \label{general-dual-riccati-function-k2}
\left\{
\begin{aligned}
&-\frac{\mathrm{d}\tilde{k}_2}{\mathrm{d}t}=-\frac{\beta}{2}\tilde{k}_2^2
+\frac{\lambda}{4} \hat{h}_{22},\indent t\le\bar{t},      \\
&\tilde{k}_2(\bar{t})=0.
\end{aligned}
\right.
\end{equation}
Subtracting \eqref{general-dual-riccati-function-k2} from \eqref{general-dual-riccati-function-k1}:
\begin{equation*} 
\left\{
\begin{aligned}
&-\frac{\mathrm{d}\left(\tilde{k}_1-\tilde{k}_2\right)}{\mathrm{d}t}
=-\frac{\beta}{2} \left(\tilde{k}_1+\tilde{k}_2\right)\left(\tilde{k}_1-\tilde{k}_2\right)+\frac{\lambda}{4}\hat{h}_{22}  \\
&\indent\indent\indent\indent\indent  -\left(H_{22}(t)-H_{33}(t)H_{13}^2(t)-\frac{\lambda}{4} \hat{h}_{22}\right)   \\
&\indent\indent\indent\indent\indent  -\left(\left(2H_{21}(t)+H_{13}^2(t)\right)\tilde{k}_1 +\frac{\beta}{2}\tilde{k}_1^2-\frac{\lambda}{4} \hat{h}_{22}\right),     \indent  t\le\bar{t}, \\
&\left(\tilde{k}_1-\tilde{k}_2\right)(\bar{t})=0.
\end{aligned}
\right.
\end{equation*}
Since $H_{22}(t)-H_{33}(t)H_{13}^2(t)$ is bounded for $\forall t\in[0,T]$, for sufficiently large $\lambda$,
{\small\begin{eqnarray*}
-\left(H_{22}(t)-H_{33}(t)H_{13}^2(t)-\frac{\lambda}{4} \hat{h}_{22}\right) \le0.
\end{eqnarray*}}
Besides, since $\frac{1}{2\beta} \left(2H_{21}(t)+H_{13}^2(t)\right)^2$ is bounded,
for sufficiently large $\lambda$, $t\le\bar{t}$,
{\small\begin{equation*}
\begin{aligned}
&-\left(\left(2H_{21}(t)+H_{13}^2(t)\right)\tilde{k}_1(t;\lambda)
+\frac{\beta}{2}\tilde{k}_1^2(t;\lambda)-\frac{\lambda}{4} \hat{h}_{22}\right) \\
&\ \ \ \ \ \ \ = -\left(-\sqrt{\frac{\beta}{2}}\tilde{k}_1 -\frac{1}{\sqrt{2\beta}} \left(2H_{21}(t)+H_{13}^2(t)\right) \right)^2  + \frac{\lambda}{4}\hat{h}_{22}+ \frac{1}{2\beta} \left(2H_{21}(t)+H_{13}^2(t)\right)^2
\le0.
\end{aligned}
\end{equation*}}
As a result, for $t\le\bar{t}$ and sufficiently large $\lambda$,
{\small\begin{eqnarray*}
&&  \frac{\lambda}{4}\hat{h}_{22}
-\left(H_{22}(t)-H_{33}(t)H_{13}^2(t)-\frac{\lambda}{4} \hat{h}_{22}\right)  \\
&&\indent \ \     -\left(\left(2H_{21}(t)+H_{13}^2(t)\right)\tilde{k}_1\left(t;\lambda\right)
+\frac{\beta}{2}\tilde{k}_1^2\left(t;\lambda\right)-\frac{\lambda}{4} \hat{h}_{22}\right)
\le \frac{\lambda}{4}\hat{h}_{22} <0.
\end{eqnarray*}}
By Lemma \ref{elementary-compare}, $ \tilde{k}_1(t;\lambda)\le\tilde{k}_2 (t;\lambda),\ t\le \bar{t}$.
Since $\frac{\lambda}{4} \hat{h}_{22}<0$, by Lemma \ref{const-term-posit-equation-posit},
$\tilde{k}_2 (t;\lambda)\le0,\ t\le \bar{t}$,
then $t_{\lambda}^{\tilde{k}_2}\le t_{\lambda}^{\tilde{k}_1}$.
It follows that, for sufficiently large $\lambda$,
$$t_{\lambda}^{\tilde{k}_2}\le t_{\lambda}^{\tilde{k}_1} \le t_\lambda^{\tilde{k}}<\bar{t}.$$
On the other hand, \eqref{general-dual-riccati-function-k2} is an equation of constant coefficients, as we have done in Lemma \ref{lemma-k-function-h22},
$$\lim_{\lambda\nearrow +\infty}t_\lambda^{\tilde{k}_2}=\bar{t},$$
whence $$\lim_{\lambda\nearrow +\infty}t_\lambda^{\tilde{k}}=\bar{t},$$
which is \eqref{limit-tilde-t-lambda-k-h22-1}.

Next, we will prove that $t_{\lambda}^{\tilde{k}}$ is increasing.
For any $\lambda_1>\lambda_2>\lambda_0(\bar{t},\tilde{k})\vee \lambda_b$,
since $h_{22}(t)<0$, we have
$$-\lambda_1 h_{22}(t)>-\lambda_2 h_{22}(t), \indent    t\in[0,T].$$
Applying  Lemma \ref{comparison theorem} to \eqref{general-dual-riccati-function},
 we get $\tilde{k}(t;\lambda_1)\le \tilde{k}(t;\lambda_2),\ t\le \bar{t}$.
Besides, by \eqref{lambda-b-h22-uniform-posit-quadratic-posit} and Lemma \ref{const-term-posit-equation-posit},
$\tilde{k}(t;\lambda)\le0,\ t\le \bar{t}$, $\lambda>\lambda_0(\bar{t},\tilde{k})\vee \lambda_b$.
Therefore,  for $\lambda_1>\lambda_2>\lambda_0(\bar{t},\tilde{k})\vee \lambda_b$,
$t_{\lambda_1}^{\tilde{k}}\geq t_{\lambda_2}^{\tilde{k}}.$
\end{proof}

\subsection{Proof of Lemma \ref{lemma-k-tilde-function-h22-cop}}
\begin{proof}[Proof of Lemma \ref{lemma-k-tilde-function-h22-cop}]
Firstly, we will prove that $t_\lambda^{\tilde{k}}$ is continuous in $\lambda\in\left(\lambda_0\left(\bar{t},\tilde{k}\right)\vee \lambda_b,+\infty\right)$.
For any $\lambda'\in\left(\lambda_0\left(\bar{t},\tilde{k}\right)\vee \lambda_b, +\infty\right)$,
the blow-up time
$t_{\lambda'}^{\tilde{k}}$ satisfies $\lim_{t\searrow t_{\lambda'}^{\tilde{k}}}\tilde{k}\left(t;\lambda\right)=+\infty$.
Consider (\ref{general-riccati-function-case}) with terminal condition $k\left(\cdot;\lambda'\right)\big|_{t=t_{\lambda'}^{\tilde{k}}}=0$:
\begin{equation} \label{general-riccati-function-case-spe-conti-de}
\left\{
\begin{aligned}
&-\frac {\mathrm{d}k}{\mathrm{d}t}=\left(2H_{21}+H_{13}^2\right)k
+H_{11}+\left(H_{22}-H_{33}H_{13}^2-\lambda' h_{22}\right)k^2, \indent  t\le t_{\lambda'}^{\tilde{k}},  \\
&k\left(\cdot;\lambda'\right)\big|_{t=t_{\lambda'}^{\tilde{k}}}=0.
\end{aligned}
\right.
\end{equation}
Then for \eqref{general-riccati-function-case-spe-conti-de},
$$-\frac{\mathrm{d}k}{\mathrm{d}t}\Big|_{t=t_{\lambda'}^{\tilde{k}}}
=H_{11}\left(t_{\lambda'}^{\tilde{k}}\right)\ge \beta.$$
By the continuity of $-\frac{\mathrm{d}k}{\mathrm{d}t}$, there is a $\delta_1>0$,
such that
$$-\frac{\mathrm{d}k}{\mathrm{d}t}\Big|_{t=t_0}>\frac{\beta}{2},\indent  \forall t_0\in\left[t_{\lambda'}^{\tilde{k}}-\delta_1, t_{\lambda'}^{\tilde{k}}+\delta_1\right].$$
Then for $\forall \epsilon_1\in(0,\delta_1)$, by Lagrangian Middle-Value Theorem, $$k\left(t_{\lambda'}^{\tilde{k}}-\epsilon_1;\lambda'\right)>\frac{\beta\epsilon_1}{2}>0.$$
Besides, the solution $k$ to \eqref{general-riccati-function-case-spe-conti-de} is continuous dependent on parameter $\lambda'$:
$$\lim_{\lambda\rightarrow\lambda'} \left|k\left(t_{\lambda'}^{\tilde{k}}-\epsilon_1;\lambda'\right)
-k\left(t_{\lambda'}^{\tilde{k}}-\epsilon_1;\lambda\right)\right|=0,$$
then there is a $\delta_{\epsilon_1}>0$, such that $\forall \lambda: \left|\lambda-\lambda'\right|<\delta_{\epsilon_1}$,
$k\left(t_{\lambda'}^{\tilde{k}}-\epsilon_1;\lambda\right)>0$.
By the definition of $t_\lambda^{\tilde{k}}$,
\begin{equation}\label{h22-tilde-t-continu-left}
t_\lambda^{\tilde{k}}>t_{\lambda'}^{\tilde{k}}-\epsilon_1.
\end{equation}
Next, choose $\bar{t}_1\in\left(t_{\lambda'}^{\tilde{k}}, \bar{t}\right)$. Recall that from Legendre transformation, $k^{-1}(t;\lambda)=\tilde{k}(t;\lambda)$ whenever both of them are not 0.
Besides, $k(t;\lambda')<0,\ t\in\left(t_{\lambda'}^{\tilde{k}}, \bar{t}_1\right]$.
Then we check (\ref{general-riccati-function-case}) with terminal condition $k(\bar{t}_1; \lambda')=\tilde{k}^{-1}\left(\bar{t}_1; \lambda'\right)$.
Solution $k(\cdot;\lambda')$ to (\ref{general-riccati-function-case}) can be extended to $[\bar{t}_2,\bar{t}_1]\supsetneqq [t^k_{\lambda'}, \bar{t}_1]$ due to its local Lipschitz coefficients.
From the continuous dependence of solution $k$ to (\ref{general-riccati-function-case}) with respect to parameter $\lambda'$,
$$\lim_{\lambda\rightarrow\lambda'} \sup_{t\in\left[\bar{t}_2,\bar{t}_1\right]}\left|k(t;\lambda')-k(t;\lambda)\right|=0.$$
Then for any sufficiently small $\epsilon_2>0$, $k(t;\lambda')$ have uniform strictly negative upper bound for $t\in\left[t^{\tilde{k}}_{\lambda'}+\epsilon_2, \bar{t}_1\right]$.
Then there is a $\delta_{\epsilon_2}>0$, such that for $\forall \lambda: |\lambda-\lambda'|<\delta_{\epsilon_2}$,
$\forall t\in\left[t^k_{\lambda'}+\epsilon_2, \bar{t}_1\right]$, $k(t;\lambda)<0$.
Then by the definition of $t_\lambda^{\tilde{k}}$,
\begin{equation}\label{h22-tilde-t-continu-right}
t_\lambda^{\tilde{k}}<t^{\tilde{k}}_{\lambda'}+\epsilon_2.
\end{equation}
By \eqref{h22-tilde-t-continu-left} and \eqref{h22-tilde-t-continu-right},
$t_\lambda^{\tilde{k}}$ is continuous in $\left(\lambda_0\left(\bar{t},\tilde{k}\right)\vee \lambda_b, +\infty\right)$.

At last, we prove that $t_\lambda^{\tilde{k}}$ is strictly increasing with respect to $\lambda$.
For $\lambda>\lambda'$,
$$-\lambda h_{22}(t)>-\lambda' h_{22}(t), \indent   t\in[0,T].$$
By Lemma \ref{comparison theorem} and Lemma \ref{const-term-posit-equation-posit},  $\tilde{k}(t;\lambda)<\tilde{k}(t;\lambda')<0,\ t<\bar{t}$.
In particular, for the above $\bar{t}_1$, $\tilde{k}(\bar{t}_1;\lambda)<\tilde{k}(\bar{t}_1;\lambda')<0$.
Then $\tilde{k}^{-1}\left(\bar{t}_1;\lambda'\right)<\tilde{k}^{-1}(\bar{t}_1;   \lambda)<0$.
Then for the following two equations:
\begin{equation*} 
\left\{
\begin{aligned}
&-\frac {\mathrm{d}k}{\mathrm{d}t}=\left(2H_{21}+H_{13}^2\right)k
+H_{11}+\left(H_{22}-H_{33}H_{13}^2-\lambda h_{22}\right)k^2, \indent  t\le\bar{t}_1,  \\
&k\left(\bar{t}_1\right)=\tilde{k}^{-1}\left(\bar{t}_1;\lambda\right),
\end{aligned}
\right.
\end{equation*}
and
\begin{equation*} 
\left\{
\begin{aligned}
&-\frac {\mathrm{d}k}{\mathrm{d}t}=\left(2H_{21}+H_{13}^2\right)k
+H_{11}+\left(H_{22}-H_{33}H_{13}^2-\lambda' h_{22}\right)k^2, \indent  t\le\bar{t}_1,  \\
&k\left(\bar{t}_1\right)=\tilde{k}^{-1}\left(\bar{t}_1;\lambda'\right),
\end{aligned}
\right.
\end{equation*}
from $\tilde{k}^{-1}\left(\bar{t}_1;\lambda'\right)<\tilde{k}^{-1}(\bar{t}_1;\lambda)<0$
and $-\lambda h_{22}(t)>-\lambda' h_{22}(t), t\in[0,T]$,
by Lemma \ref{comparison theorem}, $k(t;\lambda)>k(t;\lambda'),\ t\le \bar{t}_1$.
In particular, $k\left(t_{\lambda'}^{\tilde{k}};\lambda\right)>k\left(t_{\lambda'}^{\tilde{k}};\lambda'\right)=0$,
whence $t_{\lambda}^{\tilde{k}}>t_{\lambda'}^{\tilde{k}}$.
\end{proof}

\subsection{Review of Peng's viewpoint from Functional Analysis}\label{review-section-functionalanalysis}
Consider eigenvalue problems in the form of \eqref{evp2} with time-dependent coefficients and negative definite perturbation matrix $\bar{H}$,
then $-\bar{H}$ is positive. Denote by $M^2(0,T;\mathbb{R}^{n})$  the Hilbert space of all the $\mathbb{R}^{n}$-valued $\mathscr{F}_t$-adapted and mean square-integrable processes.
Let $g$ be the square root of $-\bar{H}$, that is, $g^2=-\bar{H}$.  Now, write the $3n\times 3n$ matrix  $g$ as  $g_{3n\times3n}=\left(-g_1^\top,g_2^\top,g_3^\top\right)$.  For any $\eta\in M^2(0,T;\mathbb{R}^{3n})$, let $\xi_\eta=(x,y,z)$ be the solution to the following FBSDE: \begin{equation*}
\left\{
\begin{aligned}
&\mathrm{d}x_t=\left[H_{2}(t)\xi_\eta+ g_{2}(t)\eta\right]\mathrm{d}t
  +\left[H_{3}(t)\xi_\eta+ g_{3}(t)\eta\right]\mathrm{d}B_t,   \indent  t\in[0,T],       \\
&-\mathrm{d}y_t=\left[H_{1}(t)\xi_\eta+  g_{1}(t)\eta\right]\mathrm{d}t-z_t \mathrm{d}B_t,  \indent t\in[0,T],  \\
&x(0)=0,\indent y(T) =0.
\end{aligned}
\right.
\end{equation*}
As given in S. Peng \cite{peng}, we can define an operator $\mathcal{A}_g: M^2(0,T;\mathbb{R}^{3n})\rightarrow M^2(0,T;\mathbb{R}^{3n})$ as
\begin{eqnarray}\label{operator-A-linear}
\mathcal{A}_g \eta=g^\top\xi_\eta.
\end{eqnarray}
Then we have the following two lemmata.
\begin{lem}[\cite{peng}, Lemma 7.2]\label{lemma-peng-7.2}
The operator $\mathcal{A}_g$ defined in (\ref{operator-A-linear}) is a linear bounded operator with the following monotonicity:
\begin{equation*}
\left\langle\mathcal{A}_g\eta,\eta\right\rangle_{M^2(0,T;\mathbb{R}^{3n})}\ge \alpha\|\mathcal{A}_g\eta\|^2_{M^2(0,T;\mathbb{R}^{3n})}.
\end{equation*}
\end{lem}

\begin{lem}[\cite{peng}, Lemma 7.3]\label{lemma-peng-7.3}
The linear operator $\mathcal{A}_g$ in (\ref{operator-A-linear}) defined on $M^2(0,T;\mathbb{R}^{3n})$ is self-adjoint:
\begin{equation*}
\langle\mathcal{A}_g\eta,\eta'\rangle_{M^2(0,T;\mathbb{R}^{3n})}
=\langle\eta,\mathcal{A}_g\eta'\rangle_{M^2(0,T;\mathbb{R}^{3n})},\indent
\forall\eta,\eta'\in M^2(0,T;\mathbb{R}^{3n}).
\end{equation*}
\end{lem}
Lemma \ref{lemma-peng-7.2} and Lemma \ref{lemma-peng-7.3}
means that the operator $\mathcal{A}_g$ is a positive operator on $M^2(0,T;\mathbb{R}^{3n})$ and it has only positive eigenvalues.  The eigenvalue problem \eqref{evp2} is closely  related to the eigenvalue problem of the operator $\mathcal{A}_g$.  In fact, for $\lambda\in \mathbb R$ such that
\begin{equation*}
\eta=\lambda\mathcal{A}_g\eta=\lambda g^\top\xi_\eta,
\end{equation*}
by the definition of $\xi_\eta$,  $  \xi_\eta$ is the solution to the following FBSDE:
\begin{equation}\label{functional-with-lambda-fbsde}
\left\{
\begin{aligned}
&\mathrm{d}x_t=\left[H_{2}(t)\xi+\lambda g_{2}(t)g^\top(t)\xi\right]\mathrm{d}t
  +\left[H_{3}(t)\xi+\lambda g_{3}(t)g^\top(t)\xi\right]\mathrm{d}B_t,  \indent t\in[0,T],        \\
&-\mathrm{d}y_t=\left[H_{1}(t)\xi+\lambda g_{1}(t)g^\top(t)\xi\right]\mathrm{d}t-z_t \mathrm{d}B_t,  \indent t\in[0,T],  \\
&x(0)=0,\indent y(T) =0.
\end{aligned}
\right.
\end{equation}
This means that $\lambda$ is an eigenvalue of problem \eqref{evp2}.

\begin{rem}\label{positive real number} By the above reasoning, $\lambda$ is an eigenvalue of \eqref{evp2} with negative perturbation $\bar{H}$ if and only if $\frac{1}{\lambda}$ is an eigenvalue of $\mathcal {A}_g$. By Lemma \ref{lemma-peng-7.2} and Lemma \ref{lemma-peng-7.3}, only a positive real number can be an eigenvalue of operator $\mathcal{A}_g$. It follows that, for negative perturbation $\bar{H}$, all the eigenvalues of \eqref{evp2} are positive.
\end{rem}

\subsection{Legendre transformation of stochastic Hamiltonian system}\label{intro-legendre-transfor-sec}
The material in this subsection is from \cite{peng}. For the convenience of readers, we give a brief introduction.
Consider the following stochastic Hamiltonian system:
\begin{equation}  \label{original-system-hamilton}
\left\{
\begin{aligned}
&\mathrm{d}x_t=\partial_yh(x_t,y_t,z_t)\mathrm{d}t
   +\partial_zh(x_t,y_t,z_t)\mathrm{d}B_t,  \indent  t\in[0,T],              \\
&-\mathrm{d}y_t=\partial_xh(x_t,y_t,z_t)\mathrm{d}t-z_t\mathrm{d}B_t, \indent  t\in[0,T],\\
& x(0)=x_0, \indent  y(T)=\partial_x\varPhi(x_T),
\end{aligned}
\right.
\end{equation}
where $h$: $\dbR^n\times \dbR^n \times \dbR^n \mapsto \dbR$ is a $C^2$
real function of $(x,y,z)$, $\varPhi$: $\dbR^n \mapsto \dbR$ is a $C^2$
real function of $x$.
We also assume that $\partial^2_{zz}h(x,y,z)\le -\beta I_n$ and $\partial^2_{xx}\varPhi(x)\ge  \beta I_n$ uniformly for $(x,y,z)\in \dbR^{3n}$.

On one hand, change the role of $(x_t, y_t)$, such that $(\tilde{x}_t, \tilde{y}_t)=(y_t, x_t)$.  On the other hand, take Legendre transformation for $h$ with respect to $z$, and for $\varPhi$ with respect to $x$ as follows:
\begin{eqnarray}
&& \tilde{h}(\tilde{x},\tilde{y},\tilde{z}) =  \inf_{z\in \mathbb{R}^n}\left\{\langle z,\tilde{z}\rangle
-h(\tilde{y},\tilde{x},z) \right\},   \notag      \\
 && \tilde{\varPhi}(\tilde{x})= \sup_{x \in \mathbb{R}^n} \left\{ \langle x,\tilde{x}\rangle -\varPhi(x)\right\}. \notag
\end{eqnarray}
The above  two steps derive a  dual Hamiltonian $\tilde{h}$ of original system \eqref{original-system-hamilton}.
Further, we have the following relations:
\begin{equation*}
 \tilde{z}=\partial_z h(\tilde{y},\tilde{x},z(\tilde{x},\tilde{y},\tilde{z})), \indent \tilde{x}=\partial_x\varPhi(x(\tilde{x})), \indent \forall \tilde{x},\tilde{y},\tilde{z}\in \mathbb{R}^n.
\end{equation*}
Most importantly, $(\tilde{x}_t,\tilde{y}_t,\tilde{z}_t)\triangleq(y_t,x_t,\partial_z h(x_t,y_t,z_t))$ satisfies the following stochastic Hamiltonian system:
\begin{equation}\label{dual-system-hamilton}
\left\{
\begin{aligned}
&\mathrm{d}\tilde{x}_t=\partial_{\tilde{y}} \tilde{h}(\tilde{x}_t,\tilde{y}_t,\tilde{z}_t)\mathrm{d}t
   +\partial_{\tilde{z}}\tilde{h}(\tilde{x}_t,\tilde{y}_t,\tilde{z}_t)\mathrm{d}B_t,  \indent  t\in[0,T],   \\
&-\mathrm{d}\tilde{y}_t=\partial_{\tilde{x}}\tilde{h}(\tilde{x}_t,\tilde{y}_t,\tilde{z}_t)\mathrm{d}t
-\tilde{z}_t\mathrm{d}B_t, \indent  t\in[0,T],\\
&\tilde{x}(0)=y_0, \indent  \tilde{y}(T)=\partial_{\tilde{x}}\tilde{\varPhi}(\tilde{x}_T).
\end{aligned}
\right.
\end{equation}
If, conversely, performing Legendre transformation to dual system \eqref{dual-system-hamilton}:
\begin{eqnarray*}
&& h(x,y,z)=\inf_{\tilde{z}\in \mathbb{R}^n}\{\langle z,\tilde{z}\rangle-\tilde{h}(y,x,\tilde{z})\},  \notag\\
&&  \varPhi (x)=\sup_{\tilde{x}\in \mathbb{R}^n}\{\langle x,\tilde{x}\rangle-\tilde{\varPhi}(\tilde{x})\}, \\
&& z=\partial_{\tilde{z}}\tilde{h}(y,x,\tilde{z}(x,y,z)), \  x=\partial_{\tilde{x}}\tilde{\varPhi}(\tilde{x}(x)), \indent  \forall x,y,z\in \mathbb{R}^n.  \notag
\end{eqnarray*}
we will reach the original Hamiltonian system \eqref{original-system-hamilton}.

In particular, for linear case, through Legendre transformation, the dual Hamiltonian $\tilde{H}$ of original Hamiltonian $H$ is
\begin{eqnarray*}
(\widetilde{H}_{ij})_{3\times3}=
\begin {bmatrix}
H_{23}H_{33}^{-1}H_{32}-H_{22}&H_{23}H_{33}^{-1}H_{31}-H_{21}&-H_{23}H_{33}^{-1}\\
H_{13}H_{33}^{-1}H_{32}-H_{12}&H_{13}H_{33}^{-1}H_{31}-H_{11}&-H_{13}H_{33}^{-1}\\
-H_{33}^{-1}H_{32}&-H_{33}^{-1}H_{31}&H_{33}^{-1}
\end{bmatrix},
\end{eqnarray*}
where each block $\widetilde{H}_{ij}, i,j=1,2,3,$ is a $n\times n$ matrix.
The relation between $z(t)$ and $\tilde{z}(t)$ is:
$$z(t)=-H_{33}^{-1}H_{32}\tilde{x}(t)-H_{33}^{-1}H_{31}\tilde{y}(t)
+H_{33}^{-1}\tilde{z}(t).$$


\end{document}